\documentclass{amsart}

\usepackage{amsmath,amssymb,amsfonts,amssymb,amsthm}

\usepackage{verbatim}
\usepackage[usenames]{color}
\usepackage{hyperref}
\usepackage{url}
\usepackage{tikz,tikz-qtree,ifthen,cancel}
\usepackage{array,tikz-qtree,ifthen,cancel}
\usepackage{graphicx}
\usepackage{adjustbox}
\usepackage{amsthm,graphicx,tikz,appendix,tikz-qtree,ifthen,cancel}
\usetikzlibrary{calc,shapes,patterns,positioning}

\newtheorem{thm}{Theorem}[section]

\newtheorem*{mainthm}{Main Theorem}

\newtheorem{lem}[thm]{Lemma}
\newtheorem{cor}[thm]{Corollary}

\newtheorem*{thm7.3}{Theorem \ref{thm.MillikenSWP}}

\newtheorem{fact}[thm]{Fact}

\theoremstyle{remark}
\newtheorem{rem}[thm]{Remark}

\theoremstyle{definition}
\newtheorem{defn}[thm]{Definition}

\newtheorem{notation}[thm]{Notation}

\newtheorem{assumption}[thm]{Assumption}

\newtheorem{question}[thm]{Question}

\theoremstyle{remark}

\newcommand{\al}{\alpha}
\newcommand{\om}{\omega}

\newcommand{\sse}{\subseteq}
\newcommand{\contains}{\supseteq}
\newcommand{\forces}{\Vdash}

\DeclareMathOperator{\ran}{ran}

\DeclareMathOperator{\depth}{depth}

\DeclareMathOperator{\Ext}{Ext}

\newcommand{\re}{\restriction}
\newcommand{\bP}{\mathbb{P}}

\newcommand{\bN}{\mathbb{N}}

\newcommand{\bR}{\mathbb{R}}

\newcommand{\bS}{\mathbb{S}}

\newcommand{\E}{\mathrm{E}}

\newcommand{\ssim}{\stackrel{s}{\sim}}

\newcommand{\ra}{\rightarrow}

\newcommand{\lgl}{\langle}
\newcommand{\rgl}{\rangle}

\newcommand{\rl}{\upharpoonleft}

\newcommand{\Erdos}{Erd{\H{o}}s}
\newcommand{\Fraisse}{Fra{\"{i}}ss{\'{e}}}
\newcommand{\Hubicka}{Hubi{\v{c}}ka}
\newcommand{\Lauchli}{L{\"{a}}uchli}

\newcommand{\Nesetril}{Ne{\v{s}}et{\v{r}}il}

\newcommand{\noprint}[1]{\relax}

\title[Borel sets of  Rado graphs and  Ramsey's Theorem]{Borel sets of  Rado graphs and  Ramsey's Theorem}

\author{Natasha Dobrinen}
\address{Department of Mathematics\\
 University of Denver \\
C.M.\ Knudson Hall, Room 300\\
2390 S.\ York St.\\ Denver, CO \ 80208 U.S.A.}
\email{natasha.dobrinen@du.edu}
  \urladdr{\url{http://web.cs.du.edu/~ndobrine}}
\thanks{This research was supported by  National Science Foundation Grants DMS-1600781 and DMS-1901753}


\begin{document}

\begin{abstract}
The well-known Galvin-Prikry Theorem  \cite{Galvin/Prikry73} states that Borel subsets of the Baire space are Ramsey: Given any  Borel subset $\mathcal{X}\sse [\om]^{\om}$, where  $[\om]^{\om}$ is endowed with the metric topology,
 each infinite subset $X\sse \om$ contains an infinite subset $Y\sse X$ such that  $[Y]^{\om}$ is either contained in $\mathcal{X}$   or disjoint from $\mathcal{X}$.
  Kechris, Pestov, and Todorcevic    point out  in \cite{Kechris/Pestov/Todorcevic05}
  the dearth of similar results for homogeneous structures.
  Such results are  a necessary step to the larger goal of
finding a correspondence between structures with infinite dimensional Ramsey properties and topological dynamics,
extending  their correspondence between the Ramsey property  and extreme amenability.
In this article, we prove an  analogue of the Galvin-Prikry theorem for
the Rado graph.
Any such infinite dimensional Ramsey theorem is
subject to
constraints  following from work in \cite{Laflamme/Sauer/Vuksanovic06}.
The proof uses
techniques developed for the author's  work on the Ramsey  theory  of the Henson graphs (\cite{DobrinenJML20} and \cite{DobrinenH_k19})
as well as some new methods for  fusion sequences,
used to bypass  the lack of
a certain amalgamation  property enjoyed by the Baire space.
\end{abstract}
\maketitle


\section{Introduction}\label{section.intro}

Ramsey theory
was initiated  by
 the following celebrated result.

\begin{thm}[Infinite Ramsey Theorem, \cite{Ramsey30}]\label{thm.RamseyInfinite}
Given positive integers $m$ and $\ell$,
suppose the collection of  all $m$-element subsets of $\om$
is partitioned into $\ell$ pieces.
Then there is an infinite  subset  $N\sse\om$
such that all $m$-element subsets of $N$  are contained in the same piece of the partition.
\end{thm}

In the arrow notation, this is written as follows:
\begin{equation}
\forall m,j\ge 1,\ \ \om\ra (\om)^m_{\ell}.
\end{equation}

One may ask whether analogues of this theorem exist when, instead of $m$-sized sets, one wants to partition  the  infinite sets of natural numbers into finitely many pieces.
Using standard set-theoretic notation,
$\om$ denotes the set of natural numbers $\{0,1,2,\dots\}$,
$[\om]^{\om}$ denotes the set of all infinite subsets of $\om$,
and given $X\in[\om]^{\om}$,
 the collection of infinite subsets of $X$ is denoted by $[X]^{\om}$.
\Erdos\ and Rado \cite{Erdos/Rado52}
showed that there is a partition of $[\om]^{\om}$ into two sets such that for each $X\in[\om]^{\om}$, the set $[X]^{\om}$ intersects both pieces of the partition.
However, this example is highly non-constructive, using the Axiom of choice to generate the partition,
and  Dana Scott suggested that all sufficiently definable sets might satisfy an infinite dimensional Ramsey analogue.
This was proven   to be the case, as we now review.

We hold to  the convention that sets of natural numbers are enumerated in increasing order, and we write $s\sqsubset X$  exactly when  $s$ is an initial segment of $X$.
The collection of finite subsets of  natural numbers  is denoted by
 $[\om]^{<\om}$.
The {\em Baire space} is  the set $[\om]^{\om}$ with the topology generated by basic open sets of the form $\{X\in [\om]^{\om}:s\sqsubset X\}$, for  $s\in[\om]^{<\om}$.
We call this the {\em metric topology} since it is  the topology generated by  the metric defined  as follows:  For distinct  $X,Y\in [\om]^{\om}$,
$\rho(X,Y)= 2^{-n}$, where $n$ is maximal such that
$X$  and $Y$
have the same
 initial segment of
 of cardinality $n$.
A subset $\mathcal{X}\sse[\om]^{\om}$ is called  {\em Ramsey} if  there is an $X\in[\om]^{\om}$ such that
either $[X]^{\om}\sse \mathcal{X}$ or else
$[X]^{\om}\cap \mathcal{X}=\emptyset$.

The first  achievement in
  the line of infinite dimensional Ramsey theory  is the result of  Nash-Williams in \cite{NashWilliams65} showing that clopen subsets  of the Baire space are Ramsey.
  Three years later,  Galvin stated in \cite{Galvin68}  that this generalizes to all open sets in the Baire space.
Soon after, the following significant  result  was proved by Galvin and Prikry.
In order to present their result, first a bit of terminology is introduced.
Given a finite set $s\in[\om]^{<\om}$ and an infinite set $X\in[\om]^{\om}$, let
\begin{equation}\label{eq.EllentuckOpen}
[s,X]=\{Y\in [X]^{\om}:s\sqsubset Y\}.
\end{equation}
A subset $\mathcal{X}\sse [\om]^{\om}$ is  called {\em completely Ramsey}
if for each finite $s$ and infinite $X$ with $s\sqsubset X$,
there is a $Y\in [s,X]$ such that either $[s,Y]\sse\mathcal{X}$ or else $[s,Y]\cap\mathcal{X}=\emptyset$.

\begin{thm}[Galvin and Prikry, \cite{Galvin/Prikry73}]\label{thm.GP}
Every Borel subset of the Baire space is completely Ramsey.
\end{thm}
It follows that   Borel sets are   Ramsey.
This weaker  statement   is written as
\begin{equation}
 \om{\xrightarrow{\mathrm{Borel}}}(\om)^{\om}.
\end{equation}

Shortly after this, Silver proved in \cite{Silver70} that analytic subsets of the Baire space are completely Ramsey.
The apex of results on  infinite dimensional Ramsey theory of the Baire space was attained by Ellentuck in \cite{Ellentuck74}.
He used the idea behind completely Ramsey sets to introduce a topology refining the metric topology on the Baire space.
In current terminology,  the topology generated by the basic open sets
 of the form $[s,X]$ in equation (\ref{eq.EllentuckOpen})
  is  called the {\em Ellentuck topology}.
Ellentuck   used this topology to  precisely characterize
those subsets of $[\om]^{\om}$ which are completely Ramsey.

\begin{thm}[Ellentuck, \cite{Ellentuck74}]\label{thm.Ellentuck}
A subset  $\mathcal{X}$ of $[\om]^{\om}$ is completely Ramsey if and only if $\mathcal{X}$ has the property of Baire in the Ellentuck topology.
\end{thm}

\begin{rem}
The notion of a {\em completely Ramsey}  subset  of the Baire space  defined above
  is due to Galvin and Prikry  and
   was used by Silver in \cite{Silver70}.
The definition of completely Ramsey used in \cite{Galvin/Prikry73} is actually slightly stronger, but we use the  form  defined above, as it is the most widely known and  provides the best analogy  for our results.
\end{rem}

Expanding now to the setting of  structures,
given a structure $\mathbb{B}$ and a substructure $\mathbb{A}$ of $\mathbb{B}$,
let ${\mathbb{B}\choose\mathbb{A}}$ denote the set of all copies of $\mathbb{A}$ in $\mathbb{B}$.
A \Fraisse\ class $\mathcal{K}$ has the {\em Ramsey property}
if   for any $\mathbb{A},\mathbb{B}\in\mathcal{K}$ with
  $\mathbb{A}$ embedding into $\mathbb{B}$,
for any  $\ell\in\om$, there is some $\mathbb{C}\in\mathcal{K}$ such that  $\mathbb{B}$ embeds into $\mathbb{C}$ and for any coloring of
  ${\mathbb{C}\choose\mathbb{A}}$
   into $\ell$ colors,
  there is some  $\mathbb{B}'\in
   {\mathbb{C}\choose\mathbb{B}}$
  all members of
   ${\mathbb{B}'\choose\mathbb{A}}$
  have the same color.
In \cite{Kechris/Pestov/Todorcevic05}, Kechris, Pestov, and Todorcevic proved  a beautiful  correspondence between
 the Ramsey property and topological dynamics:
 The group   of automorphisms of the \Fraisse\ limit
 $\mathbb{K}$ (also called a {\em \Fraisse\ structure}) of a \Fraisse\ order class $\mathcal{K}$ is
 extremely amenable
 if and only if $\mathcal{K}$ has
  the {\em Ramsey property}  (Theorem 4.7).
In Problem 11.2, they ask  for  the topological dynamics analogue
 of a    corresponding   infinite  Ramsey-theoretic   result for several \Fraisse\ structures, in particular,
  the rationals, the Rado graph, and  the Henson graphs.
By an infinite Ramsey-theoretic result,
they mean
 a result of the form
\begin{equation}\label{eq.Prob11.2}
\forall \ell\in\om,\ \
\mathbb{K}\ra_* (\mathbb{K})^{\mathbb{K}}_{\ell,t}
\end{equation}
where  equation (\ref{eq.Prob11.2})
reads:
``For each   $\ell\in \om$ and each partition of
 ${\mathbb{K}\choose\mathbb{K}}$ into $\ell$ many  definable   subsets,
there is a $\mathbb{J}\in {\mathbb{K}\choose\mathbb{K}}$
such that ${\mathbb{J}\choose\mathbb{K}}$ is  contained in
no more than $t$ of the pieces of the partition.''
Here,
one  assumes a  natural topology on  ${\mathbb{K}\choose\mathbb{K}}$ and
{\em definable} refers to
any reasonable class of sets definable relative to  the topology, for instance,
 open, Borel, analytic,  or property of Baire.
A sub-question   implicit in   Problem 11.2 in \cite{Kechris/Pestov/Todorcevic05} is the following:

\begin{question}\label{q.KPT}
For which ultrahomogeneous structures $\mathbb{K}$ is there some positive integer $t$ such that
for all $\ell\in\om$,   $\mathbb{K}\,  \ra_* \, (\mathbb{K} )^{\mathbb{K}}_{\ell,t}$?
\end{question}

This question makes sense  in view of the natural topology on  ${\mathbb{K}\choose\mathbb{K}}$ inherited as a   subspace of the Baire space.
Since the universe $K$ of
 $\mathbb{K}$   is countable, we may assume that  $K=\om$.
Letting $\iota: {\mathbb{K}\choose\mathbb{K}}\ra [\om]^{\om}$ be the map defined by $\iota(\mathbb{J})=J$, the universe of $\mathbb{J}$,
  for each
  $\mathbb{J}\in {\mathbb{K}\choose\mathbb{K}}$,
  we see that the $\iota$-image of ${\mathbb{K}\choose\mathbb{K}}$ forms a  subspace of the Baire space, $[\om]^{\om}$.
 Kechris, Pestov, and Todorcevic point out  that  very little is known about Question \ref{q.KPT}, and   immediately  move on to  discuss the  problem of big Ramsey degrees of
 \Fraisse\ structures.

 We give a  brief word about big Ramsey degrees of \Fraisse\ structures as they present   constraints towards  answering Question \ref{q.KPT}.
A  \Fraisse\ limit  $\mathbb{K}$
of a \Fraisse\ class $\mathcal{K}$ is said to have
 {\em finite big Ramsey degrees} if  for each $\mathbb{A}\in \mathcal{K}$, there is some positive integer $t$ such that  for each $\ell\ge 2$,
\begin{equation}
 \mathbb{K}\ra (\mathbb{K})^{\mathbb{A}}_{\ell,t}.
 \end{equation}
 This  is the structural analogue of the infinite Ramsey Theorem \ref{thm.RamseyInfinite}, as the copies of some finite structure are partitioned into finitely many pieces,
  and  one wants a copy of the infinite structure which meets as few of the pieces as possible.
  When such a $t$ exists for a given $\mathbb{A}$, using the notation and terminology from \cite{Kechris/Pestov/Todorcevic05}, we let  $T(\mathbb{A},\mathcal{K})$ denote the minimal such $t$ and call this the {\em big Ramsey degree} of $\mathbb{A}$ in $\mathbb{K}$.
In all known cases, the big Ramsey degree $T(\mathbb{A},\mathcal{K})$ corresponds to a  canonical partition
  of  ${\mathbb{K}\choose\mathbb{A}}$
   into $T(\mathbb{A},\mathcal{K})$ many pieces
   each of which is
{\em persistent}, meaning that for any member $\mathbb{J}$ of
  ${\mathbb{K}\choose\mathbb{K}}$, the set
   ${\mathbb{J}\choose\mathbb{A}}$ meets every piece in the partition.
 Thus, it can be useful to think of the existence of finite  big Ramsey degrees as  a structural Ramsey  theorem  where one finds some  $\mathbb{J}\in{\mathbb{K}\choose\mathbb{K}}$ so that
  ${\mathbb{J}\choose\mathbb{A}}$  achieves one color for all copies of $\mathbb{A}$ in the same piece of the canonical partition.

Big Ramsey degrees for the rationals as a linear order were studied by Sierpi\'{n}ksi, Galvin, and Laver, culminating in work of Devlin \cite{DevlinThesis}.
 The Rado graph was shown to have finite big Ramsey degrees in
 \cite{Sauer06} (extending prior work in  \cite{Pouzet/Sauer96} for edge colorings),  exact degrees  being generated in  \cite{Laflamme/Sauer/Vuksanovic06} and  calculated in \cite{Larson08}.
Finite big Ramsey degrees  were proved for ultrahomogeneous Urysohn spaces in  \cite{NVT08} and for
rationals with finitely many equivalence relations in
 \cite{Laflamme/NVT/Sauer10}.
Zucker  recently answered  Question 11.2  of Kechris, Pestov, and Todorcevic
 in \cite{Zucker19}
 in the context of big Ramsey degrees,
 finding a correspondence between big Ramsey structures (\Fraisse\ structures with big Ramsey degrees which cohere in a natural manner) and topological dynamics.

A fundamental constraint toward answering Question \ref{q.KPT} for the Rado graph,  which we denote by $\mathbb{R}$, comes from work of Sauer  in \cite{Sauer06} and its  culmination  in work by
Laflamme, Sauer, and Vuksanovic in \cite{Laflamme/Sauer/Vuksanovic06}.
In those papers, they use
 antichains
 in the
tree $\mathbb{S}$ of all finite sequences of $0$'s and $1$'s to
  represent the Rado graph.
 Letting  $\mathcal{G}$ denote the  \Fraisse\ class of finite graphs,
 it is proved in \cite{Laflamme/Sauer/Vuksanovic06}
  that for a given finite graph $\mathbb{A}$, its big Ramsey degree $T(\mathbb{A},\mathcal{G})$ equals the number of {\em strong similarity types} (see Definition \ref{def.3.1.Sauer})
 of (strongly diagonal) antichains representing $\mathbb{A}$.
In particular, these strong similarity types form a canonical partition of ${\mathbb{R}\choose\mathbb{A}}$.
The big Ramsey degrees  $T(\mathbb{A},\mathcal{G})$  grow  quickly as the number of vertices in the finite graph $\mathbb{A}$ increase (see \cite{Larson08} for these numbers).

Furthermore,
Laflamme, Sauer, and Vuksanovic
prove  in \cite{Laflamme/Sauer/Vuksanovic06} that given any
representation of $\bR$ as an antichain in $\bS$ and given any
member  $\mathbb{R}'$ of
 ${\mathbb{R}\choose\mathbb{R}}$,
 every strong similarity type  of a finite or even infinite  (strongly diagonal)  antichain in $\mathbb{S}$   embeds  into the set of nodes representing $\mathbb{R}'$.
 In particular, one can apply these  results to
construct  a Borel coloring of ${\mathbb{R}\choose\mathbb{R}}$ with $\om$ many colors each of which persists in any member of ${\mathbb{R}\choose\mathbb{R}}$.
Therefore,  any positive  answer to Question \ref{q.KPT} must restrict to
a subspace of
${\mathbb{R}\choose\mathbb{R}}$ where all members  have  the same (induced) strong similarity type.
(See Section \ref{section.strongtrees}  for more details.)
Therefore, this is the tack we must  take.

Given any Rado graph $\mathbb{R}$ with universe $\om$,
let $\mathcal{R}(\mathbb{R})$ denote the collection of all subgraphs $\mathbb{R}'\in {\mathbb{R}\choose\mathbb{R}}$ with the same (induced) strong similarity type as $\bR$.
(This space will be precisely defined at
 the end of Section \ref{section.strongRadotrees}.)
Note that $\mathcal{R}(\mathbb{R})$ is a topological space, with the topology inherited from
the $\iota$-image of $\mathcal{R}(\mathbb{R})$ as a closed subspace of
the Baire space.
For $\bR' \in\mathcal{R}(\bR)$, we let $\mathcal{R}(\bR')$ denote the subspace of those
$\bR''\in\mathcal{R}(\bR)$ which are subgraphs of $\bR'$.
This space also inherits the Ellentuck topology, and  the notion of completely Ramsey
 makes sense in this setting (see Section \ref{sec.MT}).
 The following is the main theorem of the paper.

\begin{mainthm}\label{thm.main}
Let $\mathbb{R}=(\om,E)$ be the Rado graph.
Then each Borel subset of $\mathcal{R}(\mathbb{R})$ is completely Ramsey.
In particular,
if $\mathcal{X}\sse\mathcal{R}(\mathbb{R})$ is Borel,
then for each $\mathbb{R}'\in\mathcal{R}(\mathbb{R})$,
there is a  Rado graph
 $\mathbb{R}''\in\mathcal{R}(\mathbb{R}')$  such that
$\mathcal{R}(\mathbb{R}'')$ is
either  contained in  $\mathcal{X}$,
or else is disjoint from $\mathcal{X}$.
\end{mainthm}

 Investigations into big Ramsey degrees of Henson graphs  set the stage for the work in this paper.
In January 2019, the author  finished writing  the proof that the $k$-clique-free universal ultrahomogeneous graphs have finite big Ramsey degrees in  \cite{DobrinenH_k19}, building on work for the triangle-free  case in \cite{DobrinenJML20}.
 The constructions in these papers utilized ideas
 from Milliken's topological space of strong trees \cite{Milliken81}
 and  ideas from Sauer's work on the Rado graph in  \cite{Sauer06}.
Developments unique to   \cite{DobrinenJML20} and \cite{DobrinenH_k19} include the introduction of distinguished nodes in the trees used to code specific vertices in a fixed graph
to ensure that no  $k$-cliques are ever introduced,
 and the expansion to the $k$-clique-free setting of a method of Harrington using  forcing techniques
  to give an alternate ZFC proof of the Halpern-\Lauchli\ Theorem.

  Interestingly, these ideas turned out to be useful, and in fact necessary for a satisfactory infinite dimensional Ramsey theorem for  Rado graphs,  as we shall discuss in  Section \ref{section.strongtrees}.
  In that section,
strong trees,  the  Halpern-\Lauchli\ and Milliken Theorems, and relevant ideas and results from
\cite{Laflamme/Sauer/Vuksanovic06}
 are presented   to provide the reader with some intuition for the work in this paper.
 There, we will review
  how nodes in trees can be used to represent graphs,
  the notion of  strong similarity type,  and how the work  in
 \cite{Laflamme/Sauer/Vuksanovic06}  necessitates restricting to subspaces of
 ${\mathbb{R}\choose\mathbb{R}}$ in which all members have  the same (induced) strong similarity type, as we do in  the Main Theorem.
We will also
discuss why the classical Milliken Theorem is not sufficient to provide an answer to Question \ref{q.KPT}.

 In Section \ref{section.strongRadotrees},
 we  construct  topological  spaces $\mathcal{T}_{\bR}$  of {\em  strong Rado coding trees}.
 Fixing any Rado graph $\mathbb{R}$ with universe $\om$,
 the prototype  tree  $\bS_{\mathbb{R}}$  corresponding to $\mathbb{R}=(\om,E)$
is constructed  by placing distinguished nodes $\lgl c_n:n<\om\rgl$  in  the tree $\mathbb{S}$, where $c_n$ is the node with length $n$ representing the $n$-th vertex  of  $\mathbb{R}$.
These distinguished nodes $c_n$ are called {\em coding nodes}.
The space $\mathcal{T}_{\mathbb{R}}$ consists of all subtrees of $\bS_{\mathbb{R}}$ which are
strongly similar to $\bS_{\mathbb{R}}$ as trees with coding nodes
(see Definition \ref{def.3.1.likeSauer}).
There
 is a one-to-one correspondence between the members of $\mathcal{T}_{\mathbb{R}}$ and $\mathcal{R}(\mathbb{R})$,
 which is shown at the  end of  Section \ref{section.strongRadotrees}.
 The set
 $\mathcal{T}_{\mathbb{R}}$
 will be endowed with the topology
  generated by basic open sets determined by
 finite initial subtrees of members of
 $\mathcal{T}_{\mathbb{R}}$, generating a Polish space.
This corresponds  in a simple manner to the topology on $\mathcal{R}(\mathbb{R})$.

 Given a strong Rado coding tree $T\in
 \mathcal{T}_{\mathbb{R}}$, we
 let $\mathcal{T}(T)$ denote the  subspace of all
 members of
 $\mathcal{T}_{\mathbb{R}}$
 which are
 all subtrees of $T$.
 We say that a subset
  $\mathcal{X}\sse\mathcal{T}_{\mathbb{R}}$ is {\em Ramsey} if for each
  $T\in\mathcal{T}_{\mathbb{R}}$,
  there is a subtree $S\in \mathcal{T}(T)$ such that either $\mathcal{T}(S)\sse\mathcal{X}$ or else
 $\mathcal{T}(S)\cap\mathcal{X}=\emptyset$.

\begin{thm}\label{thmmain}
For any Rado graph $\mathbb{R}=(\om,E)$,
Borel subsets of the space
$\mathcal{T}_{\mathbb{R}}$  are Ramsey.
\end{thm}

The Main Theorem will be deduced from  Theorem \ref{thmmain},  via the  homeomorphism between  $\mathcal{R}(\mathbb{R})$ and
 $\mathcal{T}_{\mathbb{R}}$, discussed at the end of Section \ref{section.strongRadotrees}.
 In fact, we shall prove that Borel  subsets of $\mathcal{R}(\bR)$ are completely Ramsey (and an even stronger property called CR$^*$) in
 Theorem \ref{thm.best}, from which we deduce Theorem \ref{thmmain} and the Main Theorem.

The basic outline of the proof of
Theorem \ref{thmmain}  is simply to prove that the collection of subsets of
$\mathcal{T}_{\mathbb{R}}$ which are Ramsey contains all open sets and is closed under complements and countable unions.
Somewhat surprisingly,
 it is the containment of all open sets that presents the largest difficulty.
 The space
$\mathcal{T}_{\mathbb{R}}$ satisfies almost all of  the four axioms presented  by Todorcevic in \cite{TodorcevicBK10}, but
the axiom \bf A.3\rm(b) fails irreparably.
(See  Chapter 5, Section 1 of \cite{TodorcevicBK10} for further details on topological Ramsey spaces and the four axioms.)
Thus, we cannot simply apply the machinery of topological Ramsey spaces to conclude the Main Theorem.
Here is where the ideas from \cite{DobrinenJML20} and \cite{DobrinenH_k19} come into play.

In Section  \ref{sec.5}  we prove in Theorem \ref{thm.matrixHL}
that colorings of level sets of   strong Rado coding trees have the Ramsey property.
Importantly, this is proved while preserving the width of some finite initial segment of a  Rado coding tree, thus serving as a surrogate for the missing  Axiom \bf A.3\rm(b).

In Section \ref{sec.MainThm},
we prove that Borel subsets of
$\mathcal{T}_{\mathbb{R}}$ are completely Ramsey.
We begin by noticing that
open sets are in one-to-one correspondence with Nash-Williams families (Definition \ref{defn.NWfamily}),
and we prove in Theorem \ref{thm.GalvinNW} that 
 all open sets are completely  Ramsey.
The specific formulation of Theorem \ref{thm.matrixHL}  enables  us to do fusion arguments.\footnote{
The proof of  Theorem \ref{thm.GalvinNW}
included in this final version of the paper was developed for 
\cite{Dobrinen_SDAP}, fixing  a glitch in our original proof given in  2019. 
Later, it  was pointed out that a similar asymmetric version of combinatorial forcing was developed by Todorcevic in notes for a graduate course in Ramsey theory in 2022.  However, those notes do not directly apply to sets of the form $[B,T]^*$, nor do they include the concluding argument in our proof after Lemma \ref{lem.decides}. }

From there, we prove that the collection of sets which are completely Ramsey are closed under complementation and countable union, this last step also relying on how we set up Theorem \ref{thm.matrixHL}
so that we can do fusion arguments without Axiom \bf  A.3\rm(b).
The proof actually achieves more, showing that all Borel subsets of $\mathcal{T}_{\mathbb{R}}$ are CR$^*$ (see Definition \ref{defn.CR}).
The translation back to $\mathcal{R}(\mathbb{R})$ in Section \ref{sec.MT} concludes the proof of the main theorem.

An  interesting quandary is whether the analogue of Ellentuck's theorem holds for the space of ordered Rado graphs $\mathcal{R}$.
A discussion of this as well as the future aim  for the ultimate answer to Question \ref{q.KPT} appears in Section \ref{section.end}.
\vskip.1in

\noindent\bf Acknowledgements. \rm
We thank A.\ Panagiotopoulos for bringing
the problem of whether the Rado graph has
infinite dimensional Ramsey theory
to our attention in 2017.
Extensive  thanks go to J.\ \Nesetril\ and J.\ \Hubicka\ for including the author in the 2016 Ramsey DocCourse.
 During that time, the author was able to make a central advance  in  \cite{DobrinenJML20} due to the  congenial and stimulating  environment.
We also thank Rapha\"{e}l Carroy and Stevo Todorcevic for discussions on this work at the 2019 Luminy Workshop on Set Theory.
Lastly, we thank the anonymous referee for their patience and suggestions for making the writing more clear.


\section{Milliken's Theorem and
constraints on the
infinite dimensional Ramsey theory of Rado graphs}\label{section.strongtrees}

Minimal background on strong trees,  the Ramsey theorems for strong trees due to Halpern-\Lauchli\ and Milliken,
and topological Ramsey spaces are set forth in this section.
These theorems provide some guidelines and intuition for our work.
For a more general exposition of this area, the reader is referred to Chapter 6 in \cite{TodorcevicBK10}.

We use standard set-theoretic notation.
The set of natural numbers
$\{0,1,2,\dots\}$ is denoted by $\om$.
Each natural number $n$  is defined to be the set of natural numbers less than $n$.
Thus,  for $n\in \om$,
 $n=\{0,\dots,n-1\}$.
 We write $n<\om$ to mean $n\in \om$.
For
$n<\om$, $2^n$ denotes the  set of all
functions from $n$ into $2$.
Such functions may be thought of as
sequences of $0$'s and $1$'s of length $n$, and
we also write $s\in 2^n$ as
\begin{equation}
s=\lgl s(0),\dots, s(n-1)\rgl=\lgl s(i):i<n\rgl.
\end{equation}
Throughout, we  shall let $\mathbb{S}$  denote $\bigcup_{n<\om}2^n$; thus,
$\mathbb{S}$ is the set of all  finite sequences of $0$'s and $1$'s.
For $s\in \mathbb{S}$, write $|s|$ to denote the domain of $s$, or equivalently, the length of $s$ as a sequence.
For $m\le |s|$, write $s\re m$ to denote the truncation of the sequence to domain $m$.
 For $s,t\in \mathbb{S}$, we write $s\sse t$ if and only if
 for some $m\le |t|$, $s=t\re m$.
 We write $s\subset t$ to denote that $s$ is a proper initial segment of $t$, meaning that $s=t\re m$ for some $m<|t|$.
The notion of tree we use is  weaker than the usual definition, but is standard for  this area.

\begin{defn}\label{def.tree}
A set of nodes
$T\sse  \mathbb{S}$ is   called a {\em tree} if  there is a set of lengths $L\sse\om$ such that  $t\in T$ implies that
$|t|\in L$ and also
for each $l\in L$ less than $|t|$,
$t\re l\in T$.
Thus,  $T$ is closed under initial segments with lengths in $L$.  We call $L$ the {\em set of levels of $T$}.
\end{defn}

\begin{defn}[Strong Subtrees]\label{def.strongtree}
A tree $T\sse \mathbb{S}$ is a {\em strong subtree}  of $\mathbb{S}$ if
 for each  $t\in T$,  $|t|\in L$  and for each $l\in L$ with $l\le |t|$,
there are   nodes   $t_0,t_1$ in $T$ such that, letting $s=t\re l$,
$t_0\contains s^{\frown}0$
and
$t_1\contains s^{\frown}1$.
Given $T$ a strong subtree of  $\mathbb{S}$, we say that  $S$ is a {\em strong subtree of $T$} if  $S$ is a strong subtree of $\mathbb{S}$ and $S$ is a subset of $T$.
We let $\mathcal{S}$ denote the set of all strong subtrees of $\mathbb{S}$.
We define a partial order $\le$ on $\mathcal{S}$ by $S\le T$ if and only if $S$ is a subtree of $T$, for $S,T\in\mathcal{S}$.
\end{defn}

For $s,t\in \mathbb{S}$, $s\wedge t$,  called the {\em meet} of $s$ and $t$,
equals the sequence  $s\re m$ where $m$ is maximal such that $s\re m=t\re m$.
Given
  $s,t\in \bS$,
 define
$s<_{\mathrm{lex}} t$ if and only if
$s$ and $t$ are incomparable  under $\sse$ and
$s(|s\wedge t|)<t(|s\wedge t|)$.
We say that a bijection $\varphi$ from a tree $T$ to another tree $S$ is a {\em tree isomorphism} if $\varphi$ preserves the tree structure and the lexicographic order of the nodes.
Given $S,T\in  \mathcal{S}$, note that there is exactly one tree isomorphism between them;
 this  is called the {\em strong tree isomorphism} between $S$ and $T$.
Given $T\in\mathcal{S}$
and $L$ its set of levels, let $\lgl l_n:n<\om\rgl$ be the increasing enumeration of $L$.
For $n<\om$,
 let $T(n)$ denote the set $\{t\in T:|t|=l_n\}$.
A {\em level set} in $T$ is a subset $X\sse T$ such that each node in $X$ has the same length; equivalently, $X\sse T(n)$ for some $n<\om$.
Any  strong tree isomorphism takes level sets to level sets.

The Halpern-\Lauchli\ Theorem is a Ramsey theorem for colorings of products of  level sets of finitely many trees.
We  present the version restricted to $\mathcal{S}$, as this is all that is needed in this article.

\begin{thm}[Halpern-\Lauchli, \cite{Halpern/Lauchli66}]\label{thm.matrixHL}
Suppose $d\ge 1$ and $T_i\in \mathcal{S}$,  for each $i<d$.
Let
\begin{equation}
c:\bigcup_{n<\om}\prod_{i<d} T_i(n)\ra 2
\end{equation}
 be given.
Then there
are an infinite set  $N=\lgl n_k:k<\om\rgl \sse \om$ and
 strong subtrees $S_i\sse T_i$
   such that for each $i<d$ and  $k<\om$,
$S_i(k)\sse T_i(n_k)$,
 and  $c$ is monochromatic on
\begin{equation}
\bigcup_{k<\om }\prod_{i<d} S_i(k).
\end{equation}
\end{thm}

The Halpern-\Lauchli\ Theorem is used to obtain
a space of strong trees with  infinite dimensional  Ramsey properties.

\begin{defn}[Milliken space]\label{defn.milliken}
The {\em Milliken space} is the triple
 $(\mathcal{S},\le,r)$,
where
$\le$ is the partial ordering of subtree  on $\mathcal{S}$
and
$r_k(T)=\bigcup_{n<k}T(n)$ is  called the  {\em $k$-th restriction of  $T$}.
\end{defn}

Thus, $r_k(T)$ is a finite  tree with $k$ many levels.
We refer to such trees  as finite strong trees, and we let $\mathcal{AS}$ denote the set of all finite strong trees.
The letters $A,B,C,\dots$ will denote finite subsets of $\bS$,
as  is the custom in  Section 6.2 of \cite{TodorcevicBK10} and in \cite{DobrinenJML20}.
For $A\in\mathcal{AS}$ and $T\in\mathcal{S}$,
let
\begin{equation}
[A,T]=\{S\in\mathcal{S}:\exists k\, (r_k(S)=A)\mathrm{\ and\ }S\le T\}.
\end{equation}
The  sets of the form $[A,\mathbb{S}]$, $A\in \mathcal{AS}$,
generate  the {\em metric topology}
 on $\mathcal{S}$, similarly to the metric topology on the Baire space.
The sets of the form $[A,T]$, $A\in\mathcal{AS}$ and $T\in\mathcal{S}$,
generate a finer  topology  on $\mathcal{S}$, analogous to the Ellentuck topology on the Baire space;
this is called  the {\em Ellentuck topology} on $\mathcal{S}$.
The Halpern-\Lauchli\ Theorem is central to the proof of the next theorem.

\begin{thm}[Milliken, \cite{Milliken81}]\label{thm.M}
If $\mathcal{X}\sse\mathcal{S}$ has the property of Baire in the Ellentuck topology on $\mathcal{S}$,
then for each basic open set $[A,T]$, where  $A\in\mathcal{AS}$ and $T\in\mathcal{S}$,
there is an $S\in [A,T]$ such that either $[A,S]\sse\mathcal{X}$ or else $[A,S]\cap\mathcal{X}=\emptyset$.
\end{thm}

This states that subsets of $\mathcal{S}$ with the property of Baire in the finer topology are completely Ramsey.
In the current terminology set forth in \cite{TodorcevicBK10},
we say that  Milliken's space of strong trees forms a topological Ramsey space.
 As a special case,  this implies
 \begin{equation}
 \forall \ell\in\om,\ \
\mathbb{S}\, {\xrightarrow{\mathrm{Borel}}}\ (\mathbb{S})^{\mathbb{S}}_{\ell,1},
 \end{equation}
  where ${\mathbb{S}\choose\mathbb{S}}=\mathcal{S}$.

Harrington came up with a novel proof of the Halpern-\Lauchli\ Theorem which uses the method of forcing to achieve a ZFC result.
The proof was known in certain circles, but not widely available until a version appeared in
\cite{Farah/TodorcevicBK}; this  proof  utilizes  a smaller uncountable cardinal, which benefits inquiries into minimal hypotheses, but at the   expense of a more complex proof.
A simpler  version closer to Harrington's original proof appears in \cite{DobrinenRIMS17},
filling in an outline provided to the author by Laver in 2011, at which time she was unaware of the proof in \cite{Farah/TodorcevicBK}.
Harrington's ``forcing proof''
uses the language and machinery of forcing to prove the existence of
finitely many finite level sets whose product  is monochromatic.
Since these objects are finite, they must be in the ground model.
This is iterated infinitely many times to the strong subtrees in the conclusion of Theorem \ref{thm.matrixHL}.
These ideas were utilized in \cite{DobrinenJML20} and \cite{DobrinenH_k19} and will be utilized again in Section \ref{sec.5}.
The difference is that we will be working with trees with special nodes to code vertices of graphs, as discussed in the next section, and so the forcing partial order must be tailored to this set-up.


\subsection{Constraints  on   infinite dimensional Ramsey theory of Rado graphs}\label{subsec.constraints}

In the work on big Ramsey degrees of the Rado graph in \cite{Sauer06} and \cite{Laflamme/Sauer/Vuksanovic06}, the authors
use an antichain of nodes in $\mathbb{S}$ to represent  a copy of the Rado graph.
The representation uses the idea of passing numbers to code the edge/non-edge relation.
Given $s,t\in \mathbb{S}$ with $|s|<|t|$,
if the nodes $s$ and $t$ represent vertices in a graph, then they represent an edge between those vertices if and only if $t(|s|)=1$.
This number $t(|s|)$ is called the {\em passing number} of $t$ at $s$.
An antichain  $A\sse \mathbb{S}$ is called {\em strongly diagonal} if its meet closure has no two nodes of the same length, and the passing number of any node at a splitting node is $0$, except of course for  extensions of the splitting node itself.
Antichains are particularly useful for coding graphs, since all nodes having  different lengths makes the passing numbers, and hence the graph being coded, clear.

Sauer introduced the notion of strong similarity map in \cite{Sauer06}, where he found  upper bounds for the big Ramsey degrees of the Rado graph  given by the number of strong similarity types of
strongly diagonal antichains representing a given finite graph.

\begin{defn}[Strong similarity, \cite{Sauer06}]\label{def.3.1.Sauer}
Let $S$ and $T$ be meet-closed subsets of $\bS$.
A function $f:S\ra T$ is a {\em strong similarity} of $S$ to $T$ if
$f$ is a bijection and
for all nodes $s,t,u,v\in S$, the following hold:
\begin{enumerate}

\item
$f$ preserves lexicographic order:
 $s <_{\mathrm{lex}}  t$ if and only if $f(s) <_{\mathrm{lex}}  f(t)$.

\item
$f$ preserves initial segments:
$s\wedge t\sse u\wedge v$ if and only if $f(s)\wedge f(t)\sse f(u)\wedge f(v)$.

\item
$f$ preserves meets:
$f(s\wedge t)=f(s)\wedge f(t)$.

\item
$f$ preserves relative lengths:
$|s\wedge t|<|u\wedge v|$ if and only if
$|f(s)\wedge f(t)|<|f(u)\wedge f(v)|$.

\item
$f$ preserves {\em passing numbers}:
If   $s,t\in S$ with  $|s|<|t|$,
then $f(t)(|f(s)|)=t(|s|)$.
\end{enumerate}
We  say that $S$ and $T$ are {\em strongly similar}  and  write $S\ssim T$  exactly when there is a strong similarity between $S$ and $T$.
We call the equivalence classes of strongly similar meet-closed sets {\em strong similarity types}.
\end{defn}

The number of strong similarity types of strongly diagonal antichains
representing a finite graph $\mathbb{A}$
were
 proved to be the exact big Ramsey degree
 $T(\mathbb{A},\mathcal{G})$
 in \cite{Laflamme/Sauer/Vuksanovic06}.
 It follows from Theorem 4.1  in
  that paper that given any   strongly diagonal antichain $U\sse\mathbb{S}$  representing  the Rado graph,
each  strong similarity type  of  any strongly diagonal antichain, finite or infinite, embeds into $U$.
Using this result, one can construct an open coloring of  all subcopies of the Rado graph (${\mathbb{R}\choose\mathbb{R}}$ as a topological subspace of $[\om]^{\om}$) into $\om$ many colors such that  for any $\mathbb{R}'\in {\mathbb{R}\choose\mathbb{R}}$,
all of the colors appear in ${\mathbb{R}'\choose\mathbb{R}}$.
Thus, any positive  answer to Question \ref{q.KPT}
must  restrict to a subspace of
 ${\mathbb{R}\choose\mathbb{R}}$ in which all members have the same strong similarity type.
 This is the approach we present in the next section.



Since Laflamme, Sauer, and Vuksanovic used Milliken's Theorem to find the big Ramsey degrees of the Rado graph,
it is natural to ask whether
it can be used to answer  Question \ref{q.KPT} for the Rado graph.
In fact, one can obtain a nominal infinite dimensional Ramsey theory for Rado graphs using Milliken's Theorem, but nothing similar to the Galvin-Prikry Theorem, as we now review.

Each  strong tree $T\in\mathcal{S}$ codes a universal graph, say $\mathbb{U}$,  in the following way:
Let each node in $T$ represent a vertex.
Two nodes $s,t\in T$ code an edge between the vertices they represent  if and only if $|s|\ne|t|$ and the longer node has passing number $1$ at the shorter node.
As was pointed out in \cite{Sauer06},
this $\mathbb{U}$ is universal, so $\mathbb{U}$ embeds into  the Rado graph and the Rado graph embeds into $\mathbb{U}$; however, they are not isomorphic.
In this way, Milliken's Theorem can be interpreted as an infinite-dimensional  Ramsey theorem  on    subcopies of $\mathbb{U}$ whose tree representations have  the same strong similarity type  as $\mathbb{S}$.
However,
this is  not the same as coloring copies of a Rado graph.

Now one can use Milliken's Theorem to achieve the following  sort of
soft
infinite dimensional Ramsey theorem on Rado graphs:
Let $\mathbb{R}$ be the Rado graph with universe $\om$.
Let the nodes in $\mathbb{S}$ represent a copy of $\mathbb{U}$ as an induced  subgraph of $\mathbb{R}$.
Given a strong similarity type $\tau$  of a representation of the Rado graph inside of $\mathbb{S}$,
and
given a  (Baire measurable with respect to the Ellentuck topology on the Milliken space)
 coloring of the members of $\tau$ into finitely many colors,
there is a strong subtree $S\in\mathcal{S}$ such that all members of $\tau$ contained inside $S$ have the same color.
Indeed, this is a special case of   Theorem 6.13 in \cite{TodorcevicBK10}.
Then one can pull out a representation of the Rado graph $R\sse S$, and all Rado graphs represented by subsets of $R$ with strong similarity type $\tau$ will have the same color.
One can even repeat this process for finitely many different strong similarity types, and afterward, pull out a representation of the Rado graph so as to have one color on each of those strong similarity types.

This might seem at first   like a satisfactory solution to Question \ref{q.KPT}.
However, it is quite far from
providing analogues of the Galvin-Prikry or Milliken Theorems, as it
 fails   a``density'' property, precluding any hope of proving that
 definable partitions are   completely Ramsey.
By  the ``density'' property, we mean that
given a Rado graph  $\bR$, a strong similarity type $\tau$,   and a definable coloring on ${\bR\choose\bR}$, one ought to be able to fix any $\bR'\in {\bR\choose\bR}$ in $\tau$ and find
a subgraph $\bR''\in {\bR'\choose\bR}$,  again in $\tau$, so that the members of ${\bR''\choose\bR}$ in $\tau$ all have the same color.
However,
if one fixes a
Rado subgraph $\mathbb{R}'$ of $\bR$ represented by some nodes $\lgl s_n:n<\om\rgl$ within $\bS$, an application of Milliken's Theorem may provide a
strong subtree $S$
which does not contain any of those nodes $s_n$;
so while there will be subsets of $S$ coding Rado graphs, none of them will represent a subgraph of  $\bR'$.
For this same reason, one cannot conclude from Theorem 6.13 in \cite{TodorcevicBK10}   that Borel  subsets of ${\mathbb{R}\choose\mathbb{R}}$
(even restricting to those with the same strong similarity type)
are
 completely Ramsey.
These failures hinge on the fact that  there is no hard-coded representation of the Rado graph inside strong trees.

Our way around these  undesired issues is to use special distinguished  nodes to represent the vertices in a fixed Rado graph.
Then it is always clear which subgraph a strong  subtree represents.
This will allow us to prove the Main Theorem as well as the stronger Theorem \ref{thm.best}
showing that Borel subsets of $\mathcal{R}(\bR)$ are CR$^*$ and hence,  completely Ramsey.


\subsection{Brief introduction to topological Ramsey spaces}\label{subsec.tRs}

The Ellentuck space, mentioned in the Introduction, is the prototype for all topological Ramsey spaces.
After Ellentuck's theorem, many spaces with similar properties were built,
including the Milliken space $(\mathcal{S},\le,r)$ of strong trees.
These were first abstracted by
Carlson and Simpson in \cite{Carlson/Simpson90},
and their approach was refined by
Todorcevic, who  distilled  the key properties of the Ellentuck space into the  four axioms below.
The  rest of this subsection is taken almost verbatim from Section 5.1 in \cite{TodorcevicBK10}, with a few modifications and explanations tailored to this paper.

One assumes a triple
$(\mathcal{R},\le,r)$
of objects with the following properties.
$\mathcal{R}$ is a nonempty set,
$\le$ is a quasi-ordering on $\mathcal{R}$,
 and $r:\mathcal{R}\times\om\ra\mathcal{AR}$ is a mapping giving us the sequence $(r_n(\cdot)=r(\cdot,n))$ of approximation mappings, where
$\mathcal{AR}$ is  the collection of all finite approximations to members of $\mathcal{R}$.
For $a\in\mathcal{AR}$ and $X,Y\in\mathcal{R}$,
\begin{equation}
[a,Y]=\{X\in\mathcal{R}:X\le Y\mathrm{\ and\ }(\exists n)\ r_n(X)=a\}.
\end{equation}

\begin{enumerate}
\item[\bf A.1]\rm
\begin{enumerate}
\item
$r_0(X)=\emptyset$ for all $X\in\mathcal{R}$.\vskip.05in
\item
$X\ne Y$ implies $r_n(X)\ne r_n(Y)$ for some $n$.\vskip.05in
\item
$r_n(X)=r_m(Y)$ implies $n=m$ and $r_k(X)=r_k(Y)$ for all $k<n$.\vskip.1in
\end{enumerate}
\end{enumerate}

For $a\in\mathcal{AR}$, let $|a|$ denote the length of the sequence $a$.
Thus, $|a|$ equals the integer $k$ for which $a=r_k(X)$ for some $X\in\mathcal{R}$.
For $a,b\in\mathcal{AR}$,  we write $a\sqsubseteq b$ if and only if  there are $X\in\mathcal{R}$ and $m\le n$ such that $a=r_m(X)$ and $b=r_n(X)$.
We write
$a\sqsubset b$ if and only if $a\sqsubseteq b$  and $a\ne b$.
\vskip.1in

\begin{enumerate}
\item[\bf A.2]\rm
There is a quasi-ordering $\le_{\mathrm{fin}}$ on $\mathcal{AR}$ such that\vskip.05in
\begin{enumerate}
\item
$\{a\in\mathcal{AR}:a\le_{\mathrm{fin}} b\}$ is finite for all $b\in\mathcal{AR}$,\vskip.05in
\item
$X\le Y$ iff $(\forall n)(\exists m)\ r_n(X)\le_{\mathrm{fin}} r_m(Y)$,\vskip.05in
\item
$\forall a,b,c\in\mathcal{AR}[a\sqsubset b\wedge b\le_{\mathrm{fin}} c\ra\exists d\sqsubset c\ a\le_{\mathrm{fin}} d]$.\vskip.1in
\end{enumerate}
\end{enumerate}

The number $\depth_Y(a)$ is the least $n$, if it exists, such that $a\le_{\mathrm{fin}}r_n(Y)$.
If such an $n$ does not exist, then we write $\depth_Y(a)=\infty$.
If $\depth_Y(a)=n<\infty$, then $[\depth_Y(a),Y]$ denotes $[r_n(Y),Y]$.

\begin{enumerate}
\item[\bf A.3] \rm
\begin{enumerate}
\item
If $\depth_Y(a)<\infty$ then $[a,X]\ne\emptyset$ for all $X\in[\depth_Y(a),Y]$.\vskip.05in
\item
$X\le Y$ and $[a,X]\ne\emptyset$ imply that there is $X'\in[\depth_Y(a),Y]$ such that $\emptyset\ne[a,X']\sse[a,X]$.\vskip.1in
\end{enumerate}
\end{enumerate}

For each $n<\om$, $\mathcal{AR}_n=\{r_n(X):X\in\mathcal{R}\}$.
Given $X\in\mathcal{R}$,
$\mathcal{AR}_n(X)$ denotes the set $\{r_n(Y):Y\le X\}$; that is, $\mathcal{AR}_n$ relativized to $X$.
If $n>|a|$, then  $r_n[a,X]$ denotes the collection of all $b\in\mathcal{AR}_n(X)$ such that $a\sqsubset b$.

\begin{enumerate}
\item[\bf A.4]\rm
If $\depth_Y(a)<\infty$ and if $\mathcal{O}\sse\mathcal{AR}_{|a|+1}$,
then there is $X\in[\depth_Y(a),Y]$ such that
$r_{|a|+1}[a,X]\sse\mathcal{O}$ or $r_{|a|+1}[a,X]\sse\mathcal{O}^c$.\vskip.1in
\end{enumerate}

The  {\em Ellentuck topology} on $\mathcal{R}$ is the topology generated by the basic open sets
$[a,X]$;
it extends the usual metrizable topology on $\mathcal{R}$ when we consider $\mathcal{R}$ as a subspace of the Tychonoff cube $\mathcal{AR}^{\bN}$.
Given the Ellentuck topology on $\mathcal{R}$,
the notions of nowhere dense, and hence of meager are defined in the natural way.
We  say that a subset $\mathcal{X}$ of $\mathcal{R}$ has the {\em property of Baire} if and only if $\mathcal{X}=\mathcal{O}\cap\mathcal{M}$ for some Ellentuck open set $\mathcal{O}\sse\mathcal{R}$ and Ellentuck meager set $\mathcal{M}\sse\mathcal{R}$.
A subset $\mathcal{X}$ of $\mathcal{R}$ is {\em completely Ramsey} if for every $\emptyset\ne[a,X]$,
there is a $Y\in[a,X]$ such that $[a,Y]\sse\mathcal{X}$ or $[a,Y]\cap\mathcal{X}=\emptyset$.
$\mathcal{X}\sse\mathcal{R}$ is {\em completely Ramsey null} if for every $\emptyset\ne [a,X]$, there is a $Y\in[a,X]$ such that $[a,Y]\cap\mathcal{X}=\emptyset$.

\begin{rem}
In \cite{TodorcevicBK10}, Todorcevic omits  the word ``completely'' and  calls such sets simply {\em Ramsey} and {\em Ramsey null}.
In this paper, we  use   the terminology of  Galvin and Prikry in \cite{Galvin/Prikry73}, as the main theorem of this paper provides an analogue of their theorem.
\end{rem}

\begin{defn}[\cite{TodorcevicBK10}]\label{defn.5.2}
A triple $(\mathcal{R},\le,r)$ is a {\em topological Ramsey space} if every subset of $\mathcal{R}$  with the property of Baire  is  completely Ramsey and if every meager subset of $\mathcal{R}$ is completely  Ramsey null.
\end{defn}

The following result can be found as Theorem
5.4 in \cite{TodorcevicBK10}.

\begin{thm}[Abstract Ellentuck Theorem]\label{thm.AET}\rm \it
If $(\mathcal{R},\le,r)$ is closed (as a subspace of $\mathcal{AR}^{\bN}$) and satisfies axioms {\bf A.1}, {\bf A.2}, {\bf A.3}, and {\bf A.4},
then every  subset of $\mathcal{R}$ with the property of Baire is  completely Ramsey,
and every meager subset is completely Ramsey null;
in other words,
the triple $(\mathcal{R},\le,r)$ forms a topological Ramsey space.
\end{thm}

The Ellentuck space of course satisfies these four axioms with
$\mathcal{R}=[\om]^{\om}$, the partial ordering  $\le $  being $\sse$, and
the $n$-th approximation to an infinite set $X$ of natural numbers being $r_n(X)=\{x_i:i<n\}$, where
 $\{x_i:i<\om\}$ enumerates $X$ in increasing order.
 Here, $\le_{\mathrm{fin}}$ is simply the partial order $\sse$.

 Milliken's space $(\mathcal{S},\le,r)$ of strong subtrees of $\mathbb{S}$   also  forms a topological Ramsey space, by Theorem \ref{thm.M}.
 The restriction map  presented in Definition \ref{defn.milliken}
yields  that for $T\in\mathcal{S}$,
$r_n(T)$  is the finite tree consisting of the first $n$ levels of $T$.
For finite trees $A,B\in\mathcal{AS}$,
we write $A\le_{\mathrm{fin}}B$ if and only if $A$ is a subtree of $B$ and the maximal nodes in $A$ are also maximal in $B$.
(Recall that we will be using $A,B,C,\dots$ to denote  finite subsets of  $\bS$.)

\begin{rem}
Topological Ramsey space theory, especially  Milliken's space, informs our approach to answering Question \ref{q.KPT}.
However,
 Axiom \bf A.3\rm(b), an amalgamation property, fails for our space of  strong Rado coding trees,
  necessitating the work in Sections \ref{sec.5}
 and \ref{sec.MainThm}.
\end{rem}


\section{Strong Rado coding  trees}\label{section.strongRadotrees}

 Topological spaces of strong Rado coding trees
are  introduced in this section.
Recall  Definition \ref{def.tree}, the slightly looser definition of tree which is  appropriate to  the setting of  strong trees.
The  next two definitions are taken from \cite{DobrinenJML20},
in which the author developed the notion of trees with coding nodes to prove that the triangle-free Henson graph has finite big Ramsey degrees.
It turns out that these ideas are also useful for coding  homogeneous structures without  forbidden configurations, in particular, the Rado graph.
Indeed, the designated coding nodes will help us achieve the ``density'' property, discussed at the end of Subsection \ref{subsec.constraints}, which is crucial for showing that Borel sets  of strong Rado coding trees are completely Ramsey.

\begin{defn}[\cite{DobrinenJML20}]\label{defn.treewcodingnodes}
A {\em tree with coding nodes}
is a
tree $T\sse\mathbb{S}$ along  with a unary function $c^T:N\ra T$, where
 $N\le \om$
   and $c^T:N\ra T$ is an injective  function such that  $m<n<N$ implies $|c^T(m)|<|c^T(n)|$.
\end{defn}

 The  {\em $n$-th coding node} in $T$, $c^T(n)$, will often be denoted as
$c^T_n$.
We will  use $l^T_n$ to  denote  $|c^T_n|$,
the length of $c^T_n$.
The next definition shows how nodes in trees can be used to code a graph.
This  idea  goes back to \Erdos, Hajnal, and Posa, who noticed that the edge/non-edge relation induces the lexicographic order on  any given ordered collection of vertices in a graph.
The  only difference here is that we distinguish  from the outset certain nodes to code particular vertices.

\begin{defn}[\cite{DobrinenJML20}]\label{def.rep}
A graph $\mathbb{G}=(G;E)$ with vertex set $G$ enumerated as $\lgl v_n:n<N\rgl$ ($N\le\om$) is {\em  represented}  by a tree $T$ with  coding nodes $\lgl c^T_n:n<N\rgl$
if and only if
for each pair $m<n<N$,
 $v_n\, \E\, v_m\Longleftrightarrow  c^T_n(l^T_m)=1$.
We will often simply say that $T$ {\em codes} $\mathbb{G}$.
The number $c^T_n(l_m^T)$ is called the {\em passing number} of $c^T_n$ at $c^T_m$.
\end{defn}

Before we define  topological  spaces of strong Rado coding trees, we  extend the definition of strong similarity type  to trees with coding nodes.
The following appears as Definition 4.9 in \cite{DobrinenJML20}; it is simplified for the setting of this paper.

\begin{defn}\label{def.3.1.likeSauer}
Let $S,T$ be trees with coding nodes.
A function $f:S\ra T$ is a {\em strong similarity} of $S$ to $T$ if
$f$ is a bijection and
for all nodes $s,t,u,v\in S$,
(1)--(5) in Definition \ref{def.3.1.Sauer} hold as well as
the following:
\begin{enumerate}
\item[(6)]
$f$ preserves coding  nodes:
$f$ maps the set of  coding nodes in $S$
onto the set of coding nodes in $T$.
\end{enumerate}
We  say that $S$ and $T$ are {\em strongly similar}  and  write $S\ssim T$  exactly when there is a strong similarity between $S$ and $T$.
A function $f$ from $S$ into itself is called a
{\em strong  similarity embedding} if  $f$ is a strong similarity from $S$ to $f[S]$.
\end{defn}

It follows from (4) and (6) that if $S$ and $T$ are strongly similar, then they have the same number $N\le\om$ of coding nodes, and for each $n<N$,  $f(c^S_n)=c^T_n$.
Note that $\ssim$ is an equivalence relation.
We call the equivalence classes {\em strong similarity types}.

\begin{defn}[Spaces of Strong Rado Coding   Trees $(\mathcal{T}_{\mathbb{R}},\le,r)$]\label{defn.RadoTrees}
Fix a Rado graph $\mathbb{R}=(\om,E)$ so that its universe is $\om$.
Use  $v_n$ to  denote $n$, the $n$-th vertex of $\mathbb{R}$, so that there is no ambiguity when we are referring to vertices.
Let $\bS_{\mathbb{R}}$ be the tree $\mathbb{S}$ with coding nodes $\lgl c^{\bS_{\mathbb{R}}}_n:n<\om\rgl$,
where  $c^{\bS_{\mathbb{R}}}_n$ is the node of length $n$ representing  $v_n$.
That is, $c^{\bS_{\mathbb{R}}}_n$ is the node  in $\bS$ of length $n$ such that for all $m<n$,
 $c^{\bS_{\mathbb{R}}}_n(m)=1$ if and only if $v_n\, E\, v_m$.

The space $\mathcal{T}_{\mathbb{R}}$ consists of
all
images of strong similarity embeddings of $\mathbb{S}_{\mathbb{R}}$ into itself.
Thus, each $T\in \mathcal{T}_{\mathbb{R}}$  is a strong coding subtree of $\mathbb{S}_{\mathbb{R}}$  such that $T\ssim
\mathbb{S}_{\mathbb{R}}$.
The members of $\mathcal{T}_{\mathbb{R}}$ are called  {\em strong Rado coding trees}, and abbreviated as {\em sRc trees}.

We partially order
 $\mathcal{T}_{\mathbb{R}}$ by    inclusion.
Thus,   for $S,T\in \mathcal{T}_{\mathbb{R}}$,  we write
$S\le T$ if and only if $S$ is a  subtree of $T$.
Define the restriction map $r$ as follows:
Given $T\in \mathcal{T}_{\mathbb{R}}$
and $k<\om$,
$r_k(T)$ is the finite subtree of $T$ consisting of all nodes in $T$ with length less than $l^T_k$.
Define
\begin{equation}
\mathcal{AT}_k=\{r_k(T):T\in\mathcal{T}_{\mathbb{R}}\},
\end{equation}
the set of all $k$-th restrictions of members of $\mathcal{T}_{\mathbb{R}}$.
Let
\begin{equation}
\mathcal{AT}=\bigcup_{k<\om}\mathcal{AT}_k,
\end{equation}
the set of all finite approximations to members of $\mathcal{T}_{\mathbb{R}}$.

For $A,B\in\mathcal{AT}$
we write
 $A\sqsubseteq B$  if and only if
 there is some $T\in\mathcal{T}_{\mathbb{R}}$ and some $j\le k$ such that
 $A=r_j(T)$ and $B=r_k(T)$.
 In this case,
  $A$ is called  an  {\em initial segment} of $B$; we also
say that $B$ {\em end-extends} $A$.
If $A\sqsubseteq B$ and $A\ne B$, then we say that $A$ is a {\em proper initial segment} of $B$ and write $A\sqsubset B$.
Furthermore,  when  a  $j$ exists such that $A=r_j(T)$,
we shall also write $A\sqsubset T$
 and call $A$ an
 {\em initial segment} of $T$.

The {\em metric topology} on $\mathcal{T}_{\mathbb{R}}$ is the topology induced by basic open cones   of the form
\begin{equation}
[A,\mathbb{S}_{\mathbb{R}}]=\{S\in \mathcal{T}_{\mathbb{R}}: \exists k\,  (r_k(S)=  A)\},
\end{equation}
 for  $A\in\mathcal{AT}$.
The {\em Ellentuck topology} on $\mathcal{T}_{\mathbb{R}}$ is induced by basic open sets of the form
\begin{equation}
[A,T]=\{S\in \mathcal{T}_{\mathbb{R}}:\exists k\,(r_k(S)=A)\mathrm{\ and\ } S\le T\},
\end{equation}
where $A\in\mathcal{AT}$ and $T\in\mathcal{T}_{\mathbb{R}}$.
Thus, the Ellentuck topology refines the metric topology.

Given $A\in\mathcal{AT}$, let $\max(A)$ denote the set of all maximal nodes in $A$.
Define the partial ordering $\le_{\mathrm{fin}}$ on $\mathcal{AT}$ as follows:
For   $A,B\in\mathcal{AT}$,
write
$A\le_{\mathrm{fin}} B$ if and only if  $A$ is a subtree of $B$ and $\max(A)\sse\max(B)$.
 Define  $\depth_T(A)$ to equal  the  $k$ such that  $A\le_{\mathrm{fin}}r_k(T)$, if it exists; otherwise, define $\depth_T(A)=\infty$.
Lastly, given  $j<k<\om$, $A\in\mathcal{AT}_j$ and $T\in\mathcal{T}_{\mathbb{R}}$, define
\begin{equation}
r_k[A,T]=\{r_k(S):S\in [A,T]\}.
\end{equation}
\end{defn}

\begin{notation}\label{notn.R<}
Coding nodes in $\bS_{\bR}$
will be notated simply as $\lgl c_n:n<\om\rgl$.
The length of $c_n$ will be denoted as $l_n$.
\end{notation}

In  the definition of $\mathbb{S}_{\bR}$,
we specified that
 $l_n=n$.
  In order to avoid confusion,
 we shall write $l_n$ when we are referring to the length of the $n$-th coding node in $\mathbb{S}_{\bR}$, so as to be consistent with the usage of $l^T_n$ for the length of the $n$-th coding node in a  tree $T\in\mathcal{T}_\bR$.

The following are  immediate consequences of the above definitions.
For any $S,T\in\mathcal{T}_{\mathbb{R}}$,
$S$ and $T$ are strongly similar, and
the strong similarity map  from $S$ to $T$ takes $c^S_n$ to $c^T_n$, for each $n<\om$.
Note that for any $T\in\mathcal{T}_{\mathbb{R}}$,
$r_0(T)$ is the empty set.

\begin{defn}\label{defn.G_TinR}
Let $\mathbb{R}=(\om,E)$ be a fixed Rado graph.
Given $T\in\mathcal{T}_{\mathbb{R}}$, let
$i_n$ denote $l^T_n$, and let
$\mathbb{G}_T$ denote  the  induced subgraph of $\bR$
with universe
$\{v_{i_n}:n\in\om\}$.
\end{defn}

Indeed,
$\mathbb{G}_T$ is the graph  represented  by the coding nodes in $T$:
The $n$-th vertex  of  $\mathbb{G}_T$ is $v^T_n=v_{i_n}$
and for $m<n$,
\begin{equation}
 v^T_n\, E\, v^T_m
 \longleftrightarrow
 c^T_n(l^T_m)=1
\longleftrightarrow
v_{i_n}\, E\, v_{i_m}.
\end{equation}

\begin{thm}\label{thm.Radocode}
Let  $\mathbb{R}=(\om,E)$ be a Rado graph.
Then the graph $\mathbb{G}_{\bS_{\bR}}$  equals  $\bR$.
Furthermore,  for each member $T\in\mathcal{T}_{\bR}$, the subgraph $\mathbb{G}_{T}$ of $\bR$ represented by $T$ is order-isomorphic to $\bR$.
\end{thm}

\begin{proof}
By definition,
$\mathbb{G}_{\bS_{\bR}}$ is the subgraph of
$\bR$ represented  by all the coding nodes in
 $\bS_{\bR}$, which is exactly $\bR$.

Given a  strong Rado coding   tree $T\in \mathcal{T}_{\bR}$,
$T$ is the image of a strong similarity embedding from $\bS_\bR$.
Hence,
$T\ssim\bS_{\bR}$, so
for each pair $m<n<\om$,
\begin{equation}
c^T_n(l^T_m)=1\, \longleftrightarrow\,  c_n(l_m)=1.
\end{equation}
Therefore,
the  graphs $\mathbb{G}_T$ and $\mathbb{G}_{\bS_{\bR}}=\bR$ are order-isomorphic.
\end{proof}

\begin{defn}\label{defn.R}
Given a Rado graph $\mathbb{R}=(\om,E)$,
let $\mathcal{R}(\mathbb{R})$ denote the set of subgraphs $\mathbb{G}_T\le \mathbb{R}$ such that $T\in\mathcal{T}_{\bR}$.
\end{defn}

\begin{rem}
Note that $\mathcal{R}(\mathbb{R})$ is a subset of
${\mathbb{R}\choose\mathbb{R}}$,
and each  member
of $\mathcal{R}(\mathbb{R})$ is
order-isomorphic to $\bR=(\om,E)$.
However, there are order-isomorphic copies of
${\mathbb{R}\choose\mathbb{R}}$  which are not members of  $\mathcal{R}(\mathbb{R})$, because there are many different strong similarity types of trees  with coding nodes which represent   order-isomorphic copies of $\bR$.
\end{rem}

\begin{defn}\label{def.G_T}
Given a  subgraph $\mathbb{G}\le \bR=(\om,E)$,
the universe $G$ of $\mathbb{G}$ is a
 set of natural numbers.
 Let $T_{\mathbb{G}}$ denote the subtree of $\bS_\bR$ induced by the coding nodes $\{c_n:n\in G\}$;
thus, the set of  levels of $T_{\mathbb{G}}$ is  $L_{\mathbb{G}}=\{l_n:n\in G\}$,
and $T_{\mathbb{G}}$ is the tree produced by taking all meets of coding nodes in $\{c_n:n\in G\}$ and then taking restrictions of the  nodes in this meet-closed set  to the levels  in $L_{\mathbb{G}}$.
\end{defn}

By Definition \ref{defn.R},
the mapping $T\mapsto \mathbb{G}_T$ produces a
 one-to-one correspondence between members of  $\mathcal{T}_{\bR}$ and   $\mathcal{R}(\bR)$.
The mapping is easily reversible:
Given $\mathbb{R}'\in \mathcal{R}(\bR)$,
note that $T_{\bR'}$ is a member of $\mathcal{T}_{\bR}$, and that
$\mathbb{G}_{T_{\mathbb{R}'}}=\mathbb{R}'$.
Conversely,
given $S\in\mathcal{T}_\bR$,
$\mathbb{G}_S$ is a member of  $\mathcal{R}(\bR)$, and
$T_{\mathbb{G}_S}=S$.

The
 space of all infinite subgraphs of $\bR=(\om,E)$ corresponds to the Baire  space $[\om]^{\om}$ by associating a subgraph of $\mathbb{G}\le \bR$ with its universe $G$, which is a subset of $\om$.
 In particular,
 the map $\iota:\mathcal{R}(\bR)\ra[\om]^{\om}$
 defined  by $\iota(\mathbb{G})=G$ identifies $\mathcal{R}(\bR)$ with its $\iota$-image,
 which is a closed subspace of $[\om]^{\om}$.
 In what follows, we identify
 $\mathcal{R}(\bR)$ with $\iota[\mathcal{R}(\bR)]$.
The basic open sets of $\mathcal{R}(\bR)$ are those of the form
 Cone$(s)\cap\mathcal{R}(\bR)$, where
Cone$(s)=\{X\in [\om]^{\om}:s\sqsubset  X\}$ for $s\in [\om]^{<\om}$.
The map $\theta:\mathcal{R}(\bR)\ra\mathcal{T}_\bR$,  defined by
 $\theta(\mathbb{G})=T_{\mathbb{G}}$, is a bijection, with the further property that  $\mathbb{G}\le\mathbb{H}$ if and only if $T_{\mathbb{G}}\le T_{\mathbb{H}}$.
Moreover, $\theta$ is a homeomorphism, since
for each $s\in [\om]^{<\om}$ which is an initial segment of a member of
$\mathcal{R}(\bR)$,
$\theta($Cone$(s))$ is open in $\mathcal{T}_\bR$.
Furthermore,
 for each $A\in\mathcal{AT}$, $\theta^{-1}([A,\bS_\bR])$ is a union of basic open sets in $\mathcal{R}(\bR)$.
Thus, results about Borel subsets of $\mathcal{T}_{\bR}$ correspond to results about Borel subsets of $\mathcal{R}(\bR)$.
This will be revisited at the end of Section \ref{sec.MainThm}.


\section{A Halpern-\Lauchli-style Theorem for strong Rado coding  trees}\label{sec.5}

Fix a Rado graph $\bR=(\om,E)$.
We shall usually drop subscripts and   let  $\bS$ denote $\bS_{\bR}$ and $\mathcal{T}$ denote $\mathcal{T}_{\bR}$.
The topological space $\mathcal{T}$ of strong Rado coding trees defined in the previous section turns out to satisfy
all but half of one
of the four axioms  of Todorcevic in \cite{TodorcevicBK10}  guaranteeing  a topological Ramsey space.
The first two axioms are easily shown to hold, and the pigeonhole principle  (Axiom \bf A.4\rm)   is  a consequence of
work in this section (see Corollary \ref{cor.A.4}).
However, the amalgamation principle (Axiom
\bf A.3\rm(b))  fails for $\mathcal{T}$,
so we cannot simply apply Todorcevic's axioms and invoke his Abstract Ellentuck Theorem to deduce infinite dimensional Ramsey theory on $\mathcal{T}$.
It is this failure of outright amalgamation that presents the interesting challenge to proving that Borel subsets in $\mathcal{T}$ have the Ramsey property.

Our approach is to
 build the infinite dimensional Ramsey theory on $\mathcal{T}$   in a similar manner as Galvin and Prikry did for the Baire space in \cite{Galvin/Prikry73}.
 However,
 even that approach  is not exactly replicable in $\mathcal{T}$, again due to lack of amalgamation.
In this section, we prove a Ramsey theorem for colorings of level sets, namely
Theorem \ref{thm.matrixHL}.
This theorem  will yield an enhanced version of  Axiom \bf A.4 \rm   strong enough to replace some uses of
  Axiom \bf A.3\rm(b) in the Galvin-Prikry proof, providing
 alternate means  for  proving that Borel  subsets of $\mathcal{T}$  are Ramsey in the next section.

Here, we mention a theorem of \Erdos\ and Rado which will
be used in the proof of the main theorem of this section.
This  theorem   guarantees cardinals large enough to have the Ramsey property for colorings with  infinitely many colors.

\begin{thm}[\Erdos-Rado]\label{thm.ER}
For $r<\om$ and $\mu$ an infinite cardinal,
$$
\beth_r(\mu)^+\ra(\mu^+)_{\mu}^{r+1}.
$$
\end{thm}

We begin setting up notation needed for
Theorem \ref{thm.matrixHL}.
Recall that for $t\in \bS$ and $l\le |t|$,
$t\re l$ denotes the initial segment of $t$ with domain $l$.
For any subset $U$ of $\bS$, finite or infinite,
we let
\begin{equation}
\widehat{U}=\{t\re l:t\in U\mathrm{\ and\ }l\le|t|\},
\end{equation}
the tree of all initial segments  of members of $U$.
For a  finite subset $A\sse\bS$,
define
\begin{equation}
l_A=\max\{|t|:t\in A\},
\end{equation}
the maximum of the lengths of  nodes in $A$.
For $l\le l_A$,
let
\begin{equation}
A\re l=\{t\re l : t\in A\mathrm{\ and\ }|t|\ge l\}
\end{equation}
 and let
\begin{equation}
A\rl l=\{t\in A:|t|< l\}\cup A\re l.
\end{equation}
Thus, $A\re l$ is a level set, while $A\rl l$ is the set of nodes in $A$ with length less than $l$ along with the truncation
to $l$ of the  nodes in $A$ of length at least
 $l$.
In particular,
$A\re l=\emptyset$ for $l>l_A$, and
 $A\rl l=A$  for  $l\ge l_A$.
If $l$ is not the length of any node in $A$, then
  $A\rl l$ will  not  be a subset  of $A$, but  it is of course a subset of $\widehat{A}$.
Let
\begin{equation}\label{eq.AThat}
\widehat{\mathcal{AT}}=\{A\rl l:A\in\mathcal{AT}\mathrm{\ and\ }l\le l_A\}.
\end{equation}
Given $T\in \mathcal{T}$,
let  $\mathcal{AT}(T)$ denote the members of $\mathcal{AT}$ which are contained in $T$.
For $k<\om$, let $\mathcal{AT}_k(T)$ denote the set of those $A\in\mathcal{AT}_k$ such that $A$ is a subtree of $T$.
 Let $L_T=\{|t|:t\in T\}$
 and  define
\begin{equation}\label{eq.AThatT}
\widehat{\mathcal{AT}}(T)=\{A\rl l:
A\in  \mathcal{AT}(T) \mathrm{\ and\ }\ l\in L_T\}.
\end{equation}
It is important  that the maximal nodes in any member of  $ \widehat{\mathcal{AT}}(T)$
have  length in $L_T$
and therefore split in $T$.
However, there are  members of $\widehat{\mathcal{AT}}(T)$ which are not  strongly similar to  $r_n(\bS)$ for any $n$, and hence are not members of $\mathcal{AT}(T)$.
These  notions can   be relativized to any $B\in \mathcal{AT}$ in place of $T$.

Parts of the following definition will be used in this section, and the rest will be used in the next section.

\begin{defn}\label{defnhatAT}
Given $T\in \mathcal{T}$ and
$B\in\widehat{\mathcal{AT}}(T)$,
letting $m$ be the least integer
for which there exists $C\in\mathcal{AT}_m$ such that $\max(C)\sqsupseteq \max(B)$,
define
\begin{equation}\label{eq.AT}
[B,T]^*=\{S\in\mathcal{T}: \max(r_m(S))\sqsupseteq\max(B)\mathrm{\ and\ }  S\le T\}.
\end{equation}
For $n\ge m$, define
\begin{equation}\label{eq.rn*}
r_n[B,T]^*=\{r_n(S):S\in [B,T]^*\},
\end{equation}
and let
\begin{equation}
r[B,T]^*=\bigcup_{n\ge m} r_n[B,T]^*.
\end{equation}
\end{defn}

We point out that
$r_{n}[B, T]^*$ defined in equation (\ref{eq.rn*})
is equal to
$\{ C\in\mathcal{AT}_{n}(T):\max(C)\sqsupseteq \max(B)\}$.
\vskip.1in

\noindent\bf{Hypotheses for Theorem \ref{thm.matrixHL}.} \rm
Let $T\in\mathcal{T}$ be  fixed.
Suppose  $D=r_n(T)$ for some $n<\om$.
Given $A\in \widehat{\mathcal{AT}}(T)$ with $\max(A)\sse\max(D)$,
let  $A^+$ denote the union of $A$  with the set of   immediate successors  in    $\widehat{T}$ of the members of $\max(A)$; thus,
\begin{equation}
A^+=A\cup \{s^{\frown}i : s\in \max(A)\mathrm{\ and\ } i\in\{0,1\}\}
\end{equation}
and    $\max(A^+)$ is a level set of nodes of length $l_A+1$.
Let $B$ denote  the subset of $A^+$ which will be end-extended to members of $\mathcal{AT}(T)$ which are colored.
We consider two  cases  for triples $(A,B,k)$,
where $B\in\widehat{\mathcal{AT}}(T)$,  $A\sqsubset B$,
and $\max(B)\sse\max(A^+)$:

\begin{enumerate}
\item[]
\begin{enumerate}
\item[\bf{Case (a).}]
$k\ge 1$,
$A\in \mathcal{AT}_k(T)$,
and
 $B=A^+$.
\end{enumerate}
\end{enumerate}

\begin{enumerate}
\item[]
\begin{enumerate}
\item[\bf{Case (b).}]
$A$ has at least one node, and
each member of $\max(A)$ has exactly one extension in $B$.
Let $k$ be the integer satisfying $2^k=\mathrm{card}(\max(A))$.
\end{enumerate}
\end{enumerate}
In both   cases,
we will be working with $r_{k+1}[B, T]^*$.
It can be useful to notice that this set  $r_{k+1}[B, T]^*$ is equal to
$\{ C\in\mathcal{AT}_{k+1}(T):\max(C)\sqsupseteq \max(B)\}$.

In Case (a),   $A$ is actually a member of
$\mathcal{AT}_k(T)$.
In Case (b), $A$ may or may not be a member of
$\mathcal{AT}$;
it is possible that there are several different truncations of $A$ which are members of
$\mathcal{AT}$.

\begin{thm}\label{thm.matrixHL}
Let $T,D, A, B,k$ be as in one of Cases (a) or (b) in the Hypotheses above.
Let
  $h:  r_{k+1}[B,T]^*\ra 2$  be a coloring.
Then  there is a Rado tree $S\in [D,T]$ such that
$h$ is monochromatic on $ r_{k+1}[B,S]^*$.
\end{thm}

\begin{proof}
Assume the hypotheses.
Given $U\in \mathcal{AT}\cup \mathcal{T}$ with $U\sse T$, define
\begin{equation}
\Ext_U(B)=\{\max(C):C\in r_{k+1}[B,T]^*\mathrm{\ and \ } C\sse U\}.
\end{equation}
The coloring $h$ induces a coloring $h': \Ext_T(B)\ra 2$ by  defining $h'(X)=h(C_X)$, where
$C_X$ is the member of $\mathcal{AT}$
 induced by the meet-closure of $X$.

Let $d+1$ be the number of  nodes in  $\max(B)$,
and  fix an enumeration   $s_0,\dots, s_d$  of
 the nodes in    $\max(B)$
 with the property that  for any
 $X\in \Ext_T(B) $,
the coding node in $X$ extends $s_d$.
 Note that  in both Cases (a) and (b), $d+1=2^k$, as any $C\in\mathcal{AT}_{k+1}$ has $2^k$ maximal nodes.
Let $L$ denote the collection of all  $l<\om$  for which  there is a member of
 $\Ext_T(B)$ with  nodes of length $l$.

For
  $i\le d$,   let  $T_i=\{t\in T:t\contains s_i\}$.
Let $\kappa=\beth_{2d}$, so that the partition relation $\kappa\ra (\aleph_1)^{2d}_{\aleph_0}$ holds by the \Erdos-Rado Theorem \ref{thm.ER}.
The following forcing notion $\bP$    adds $\kappa$ many paths through  each  $T_i$,  $i< d$,
and one path  through $T_d$.
However, as our goal is to find a tree $S\in [D,T]$ for which $h$ is monochromatic on $r_{k+1}[B,S]^*$,
the forcing will be applied in finite increments to construct $S$, without ever moving to a generic extension.
\vskip.1in

Define $\bP$ to consist of  the set of  finite   functions $p$
of the form
$$
p:(d\times\vec{\delta}_p)\cup\{d\}\ra \bigcup_{i\le d} T_i\re l_p,
$$
where $\vec{\delta}_p\in[\kappa]^{<\om}$,
 $l_p\in L$,
 $\{p(i,\delta) : \delta\in  \vec{\delta}_p\}\sse  T_i\re l_p$ for each $i<d$,  and
$p(d)$ is {\em the} coding  node in $T\re l_p$ extending $s_d$.
The partial
 ordering on $\bP$  is defined as follows:
$q\le p$ if and only if
$l_q\ge l_p$, $\vec{\delta}_q\contains \vec{\delta}_p$,
$q(d)\contains p(d)$,
and
$q(i,\delta)\contains p(i,\delta)$ for each $(i,\delta)\in d\times \vec{\delta}_p$.

Given $p\in\bP$,
 the {\em range of $p$} is
$$
\ran(p)=\{p(i,\delta):(i,\delta)\in d\times \vec\delta_p\}\cup \{p(d)\}.
$$
If  $q\in \bP$ and $\vec{\delta}_p\sse \vec{\delta}_q$, define
$$
\ran(q\re \vec{\delta}_p)=\{q(i,\delta):(i,\delta)\in d\times \vec{\delta}_p\}\cup \{q(d)\}.
$$
Thus, $q\le p$ if and only if $\vec\delta_q\contains \vec\delta_p$ and $\ran(q\re\vec\delta_p)$ end-extends $\ran(p)$.

For  $(i,\al)\in d\times \kappa$,
 let
\begin{equation}
 \dot{b}_{i,\al}=\{\lgl p(i,\al),p\rgl:p\in \bP\mathrm{\  and \ }\al\in\vec{\delta}_p\},
\end{equation}
 a $\bP$-name for the $\al$-th generic branch through $T_i$.
Let
\begin{equation}
\dot{b}_d=\{\lgl p(d),p\rgl:p\in\bP\},
\end{equation}
a $\bP$-name for the generic branch through $T_d$.
 Given a generic filter   $G\sse \bP$, notice that
 $\dot{b}_d^G=\{p(d):p\in G\}$,
 which  is a cofinal path of coding  nodes in $T_d$.
Let $\dot{L}_d$ be a $\bP$-name for the set of lengths of coding  nodes in $\dot{b}_d$, and note that
$\bP$ forces  that $\dot{L}_d\sse L$.
Let $\dot{\mathcal{U}}$ be a $\bP$-name for a non-principal ultrafilter on $\dot{L}_d$.
Given  $p\in \bP$, notice  that
\begin{equation}
 p\forces  \forall
 (i,\al)\in d\times \vec\delta_p\, (\dot{b}_{i,\al}\re l_p= p(i,\al)) \wedge
( \dot{b}_d\re l_p=p(d)).
\end{equation}

We will write sets $\{\al_i:i< d\}$ in $[\kappa]^d$ as vectors $\vec{\al}=\lgl \al_0,\dots,\al_{d-1}\rgl$ in strictly increasing order.
For $\vec{\al}\in[\kappa]^d$,  let
\begin{equation}
\dot{b}_{\vec{\al}}=
\lgl \dot{b}_{0,\al_0},\dots, \dot{b}_{d-1,\al_{d-1}},\dot{b}_d\rgl.
\end{equation}
For  $l<\om$,
 let
 \begin{equation}
 \dot{b}_{\vec\al}\re l=
 \lgl \dot{b}_{0,\al_0}\re l,\dots, \dot{b}_{d-1,\al_{d-1}}\re l,\dot{b}_d\re l\rgl.
 \end{equation}
Using these abbreviations, one sees that
$h'$ is a coloring on level sets  of the form $\dot{b}_{\vec\al}\re l$
whenever this is
 forced to be a member of $\Ext_T(B)$.
Given $\vec{\al}\in [\kappa]^d$ and  $p\in \bP$ with $\vec\al\sse\vec{\delta}_p$,
let
\begin{equation}
X(p,\vec{\al})=\{p(i,\al_i):i<d\}\cup\{p(d)\}.
\end{equation}
Notice that
$X(p,\vec{\al})$
 is a member of
  $\Ext_T(B)$.

For each $\vec\al\in[\kappa]^d$,
choose a condition $p_{\vec{\al}}\in\bP$ satisfying the following:
\begin{enumerate}
\item
 $\vec{\al}\sse\vec{\delta}_{p_{\vec\al}}$.
\item
There is an $\varepsilon_{\vec{\al}}\in 2$
 such that
$p_{\vec{\al}}\forces$
``$h'(\dot{b}_{\vec{\al}}\re l)=\varepsilon_{\vec{\al}}$
for $\dot{\mathcal{U}}$ many $l$ in $\dot{L}_d$''.
\item
$h'(X(p_{\vec\al},\vec{\al}))=\varepsilon_{\vec{\al}}$.
\end{enumerate}
Such conditions can be found as follows:
Fix some $\tilde{X}\in \Ext_T(B)$ and let $t_i$ denote the node in $\tilde{X}$ extending $s_i$, for each $i\le d$.
For $\vec{\al}\in[\kappa]^d$,  define
$$
p^0_{\vec{\al}}=\{\lgl (i,\delta), t_i\rgl: i< d, \ \delta\in\vec{\al} \}\cup\{\lgl d,t_d\rgl\}.
$$
Then
 (1) will  hold for all $p\le p^0_{\vec{\al}}$,
 since $\vec\delta_{p_{\vec\al}^0}= \vec\al$.
 Next,
let  $p^1_{\vec{\al}}$ be a condition below  $p^0_{\vec{\al}}$ which
forces  $h(\dot{b}_{\vec{\al}}\re l)$ to be the same value for
$\dot{\mathcal{U}}$  many  $l\in \dot{L}_d$.
Extend this to some condition
 $p^2_{\vec{\al}}\le p_{\vec{\al}}^1$
 which
 decides a value $\varepsilon_{\vec{\al}}\in 2$
 so that
 $p^2_{\vec{\al}}$ forces
  $h'(\dot{b}_{\vec{\al}}\re l)=\varepsilon_{\vec{\al}}$
for $\dot{\mathcal{U}}$ many $l$ in $\dot{L}_d$.
Then
 (2) holds for all $p\le p_{\vec\al}^2$.
If $ p_{\vec\al}^2$ satisfies (3), then let $p_{\vec\al}=p_{\vec\al}^2$.
Otherwise,
take  some  $p^3_{\vec\al}\le p^2_{\vec\al}$  which forces
$\dot{b}_{\vec\al}\re l\in \Ext_T(B)$ and
$h'(\dot{b}_{\vec\al}\re l)=\varepsilon_{\vec\al}$
for
some $l\in\dot{L}$
 with
$l_{p^2_{\vec\al}}< l\le l_{p^3_{\vec\al}}$.
Since $p^3_{\vec\al}$  forces  that $\dot{b}_{\vec\al}\re l$ equals
$\{p^3_{\vec\al}(i,\al_i)\re l:i<d\}\cup \{p^3_{\vec\al}(d)\re l\}$,
which is exactly
$X(p^3_{\vec\al}\re l,\vec\al)$,
and this level set is in the ground model, it follows that $h'(X(p^3_{\vec\al}\re l,\vec\al))
=\varepsilon_{\vec\al}$.
Let
$p_{\vec\al}$ be $p^3_{\vec\al}\re l$.
Then $p_{\vec\al}$ satisfies (1)-(3).

Let $\mathcal{I}$ denote the collection of all functions $\iota: 2d\ra 2d$ such that
for each $i<d$,
$\{\iota(2i),\iota(2i+1)\}\sse \{2i,2i+1\}$.
For $\vec{\theta}=\lgl \theta_0,\dots,\theta_{2d-1}\rgl\in[\kappa]^{2d}$,
$\iota(\vec{\theta}\,)$ determines the pair of sequences of ordinals $\lgl \iota_e(\vec{\theta}\,),\iota_o(\vec{\theta}\,)\rgl$, where
\begin{align}
\iota_e(\vec{\theta}\,)&=
\lgl \theta_{\iota(0)},\theta_{\iota(2)},\dots,\theta_{\iota(2d-2))}\rgl\cr
\iota_o(\vec{\theta}\,)&=
 \lgl\theta_{\iota(1)},\theta_{\iota(3)},\dots,\theta_{\iota(2d-1)}\rgl.
 \end{align}

We now proceed to  define a coloring  $f$ on
$[\kappa]^{2d}$ into countably many colors.
Let $\vec{\delta}_{\vec\al}$ denote $\vec\delta_{p_{\vec\al}}$,
 $k_{\vec{\al}}$ denote $|\vec{\delta}_{\vec\al}|$,
$l_{\vec{\al}}$ denote  $l_{p_{\vec\al}}$, and let $\lgl \delta_{\vec{\al}}(j):j<k_{\vec{\al}}\rgl$
denote the enumeration of $\vec{\delta}_{\vec\al}$
in increasing order.
Given  $\vec\theta\in[\kappa]^{2d}$ and
 $\iota\in\mathcal{I}$,   to reduce subscripts
 let
$\vec\al$ denote $\iota_e(\vec\theta\,)$ and $\vec\beta$ denote $\iota_o(\vec\theta\,)$, and
define
\begin{align}\label{eq.fiotatheta}
f(\iota,\vec\theta\,)= \,
&\lgl \iota, \varepsilon_{\vec{\al}}, k_{\vec{\al}}, p_{\vec{\al}}(d),
\lgl \lgl p_{\vec{\al}}(i,\delta_{\vec{\al}}(j)):j<k_{\vec{\al}}\rgl:i< d\rgl,\cr
& \lgl  \lgl i,j \rgl: i< d,\ j<k_{\vec{\al}},\ \mathrm{and\ } \delta_{\vec{\al}}(j)=\al_i \rgl,\cr
&\lgl \lgl j,k\rgl:j<k_{\vec{\al}},\ k<k_{\vec{\beta}},\ \delta_{\vec{\al}}(j)=\delta_{\vec{\beta}}(k)\rgl\rgl.
\end{align}
Fix some ordering of $\mathcal{I}$ and define
\begin{equation}
f(\vec{\theta}\,)=\lgl f(\iota,\vec\theta\,):\iota\in\mathcal{I}\rgl.
\end{equation}

By the \Erdos-Rado Theorem  \ref{thm.ER},  there is a subset $K\sse\kappa$ of cardinality $\aleph_1$
which is homogeneous for $f$.
Take $K'\sse K$ so that between each two members of $K'$ there is a member of $K$.
Given  sets of ordinals $I$ and $J$,  we write $I<J$  to mean that  every member of $I$ is less than every member of $J$.
Take  $K_i\sse K'$  be  countably infinite subsets
satisfying
  $K_0<\dots<K_{d-1}$.
The next four lemmas  are almost verbatim the Claims 3 and 4 and Lemma 5.3
in \cite{DobrinenJML20}, with small necessary changes being made.
The proofs are included here for the reader's convenience.

Fix some  $\vec\gamma\in \prod_{i<d}K_i$, and define
\begin{align}\label{eq.star}
&\varepsilon^*=\varepsilon_{\vec\gamma},\ \
k^*=k_{\vec\gamma},\ \
t_d=p_{\vec\gamma}(d),\cr
t_{i,j}&=p_{\vec{\gamma}}(i,\delta_{\vec{\gamma}}(j))\mathrm{\ for\ }
i<d,\ j<k^*.
\end{align}
We show that the values  in equation (\ref{eq.star}) are the same for any choice of
$\vec\gamma$.

\begin{lem}\label{lem.onetypes}
 For all $\vec{\al}\in \prod_{i<d}K_i$,
 $\varepsilon_{\vec{\al}}=\varepsilon^*$,
$k_{\vec\al}=k^*$,  $p_{\vec{\al}}(d)=t_d$, and
$\lgl p_{\vec\al}(i,\delta_{\vec\al}(j)):j<k_{\vec\al}\rgl
=
 \lgl t_{i,j}: j<k^*\rgl$ for each $i< d$.
\end{lem}

\begin{proof}
Let
 $\vec{\al}$ be any member of $\prod_{i<d}K_i$, and let $\vec{\gamma}$ be the set of ordinals fixed above.
Take  $\iota\in \mathcal{I}$
to be the identity function on $2d$.
Then
there are $\vec\theta,\vec\theta'\in [K]^{2d}$
such that
$\vec\al=\iota_e(\vec\theta\,)$ and $\vec\gamma=\iota_e(\vec\theta'\,)$.
Since $f(\iota,\vec\theta\,)=f(\iota,\vec\theta'\,)$,
it follows that $\varepsilon_{\vec\al}=\varepsilon_{\vec\gamma}$, $k_{\vec{\al}}=k_{\vec{\gamma}}$, $p_{\vec{\al}}(d)=p_{\vec{\gamma}}(d)$,
and $\lgl \lgl p_{\vec{\al}}(i,\delta_{\vec{\al}}(j)):j<k_{\vec{\al}}\rgl:i< d\rgl
=
\lgl \lgl p_{\vec{\gamma}}(i,\delta_{\vec{\gamma}}(j)):j<k_{\vec{\gamma}}\rgl:i< d\rgl$.
\end{proof}

Let $l^*$ denote the length of the  node $t_d$, and notice that
the  node
 $t_{i,j}$ also has length $l^*$,  for each   $(i,j)\in d\times k^*$.

\begin{lem}\label{lem.j=j'}
Given any $\vec\al,\vec\beta\in \prod_{i<d}K_i$,
if $j,k<k^*$ and $\delta_{\vec\al}(j)=\delta_{\vec\beta}(k)$,
 then $j=k$.
\end{lem}

\begin{proof}
Let $\vec\al,\vec\beta$ be members of $\prod_{i<d}K_i$   and suppose that
 $\delta_{\vec\al}(j)=\delta_{\vec\beta}(k)$ for some $j,k<k^*$.
For  $i<d$, let  $\rho_i$ be the relation from among $\{<,=,>\}$ such that
 $\al_i\,\rho_i\,\beta_i$.
Let   $\iota$ be the member of  $\mathcal{I}$  such that for each $\vec\theta\in[K]^{2d}$ and each $i<d$,
$\theta_{\iota(2i)}\ \rho_i \ \theta_{\iota(2i+1)}$.
Fix some
$\vec\theta\in[K']^{2d}$ such that
$\iota_e(\vec\theta)=\vec\al$ and $\iota_o(\vec\theta)= \vec\beta$.
Since between any two members of $K'$ there is a member of $K$, there is a
 $\vec\zeta\in[K]^{d}$ such that  for each $i< d$,
 $\al_i\,\rho_i\,\zeta_i$ and $\zeta_i\,\rho_i\, \beta_i$.
Let   $\vec\mu,\vec\nu$ be members of $[K]^{2d}$ such that $\iota_e(\vec\mu)=\vec\al$,
$\iota_o(\vec\mu)=\vec\zeta$,
$\iota_e(\vec\nu)=\vec\zeta$, and $\iota_o(\vec\nu)=\vec\beta$.
Since $\delta_{\vec\al}(j)=\delta_{\vec\beta}(k)$,
the pair $\lgl j,k\rgl$ is in the last sequence in  $f(\iota,\vec\theta)$.
Since $f(\iota,\vec\mu)=f(\iota,\vec\nu)=f(\iota,\vec\theta)$,
also $\lgl j,k\rgl$ is in the last  sequence in  $f(\iota,\vec\mu)$ and $f(\iota,\vec\nu)$.
It follows that $\delta_{\vec\al}(j)=\delta_{\vec\zeta}(k)$ and $\delta_{\vec\zeta}(j)=\delta_{\vec\beta}(k)$.
Hence, $\delta_{\vec\zeta}(j)=\delta_{\vec\zeta}(k)$,
and therefore $j$ must equal $k$.
\end{proof}

For each $\vec\al\in \prod_{i<d}K_i$, given any   $\iota\in\mathcal{I}$, there is a $\vec\theta\in[K]^{2d}$ such that $\vec\al=\iota_o(\vec\al)$.
By the second line of equation  (\ref{eq.fiotatheta}),
there is a strictly increasing sequence
$\lgl j_i:i< d\rgl$  of members of $k^*$ such that
$\delta_{\vec\gamma}(j_i)=\al_i$.
By
homogeneity of $f$,
this sequence $\lgl j_i:i< d\rgl$  is the same for all members of $\prod_{i<d}K_i$.
Then letting
 $t^*_i$ denote $t_{i,j_i}$,
 one sees that
\begin{equation}
p_{\vec\al}(i,\al_i)=p_{\vec{\al}}(i, \delta_{\vec\al}(j_i))=t_{i,j_i}=t^*_i.
\end{equation}
Let $t_d^*$ denote $t_d$.

\begin{lem}\label{lem.compat}
For any finite subset $\vec{J}\sse \prod_{i<d}K_i$,
$p_{\vec{J}}:=\bigcup\{p_{\vec{\al}}:\vec{\al}\in \vec{J}\,\}$
is a member of $\bP$ which is below each
$p_{\vec{\al}}$, $\vec\al\in\vec{J}$.
\end{lem}

\begin{proof}
Given  $\vec\al,\vec\beta\in \vec{J}$,
if
 $j,k<k^*$ and
 $\delta_{\vec\al}(j)=\delta_{\vec\beta}(k)$, then
 $j$ and $k$ must be equal, by
 Lemma  \ref{lem.j=j'}.
Then  Lemma \ref{lem.onetypes} implies
that for each $i<d$,
\begin{equation}
p_{\vec\al}(i,\delta_{\vec\al}(j))=t_{i,j}=p_{\vec\beta}(i,\delta_{\vec\beta}(j))
=p_{\vec\beta}(i,\delta_{\vec\beta}(k)).
\end{equation}
Hence,
 for all
$\delta\in\vec{\delta}_{\vec\al}\cap
\vec{\delta}_{\vec\beta}$
and  $i<d$,
$p_{\vec\al}(i,\delta)=p_{\vec\beta}(i,\delta)$.
Thus,
$p_{\vec{J}}:=
\bigcup \{p_{\vec{\al}}:\vec\al\in\vec{J}\}$
is a  function with domain $\vec\delta_{\vec{J}}\cup\{d\}$, where
$\vec\delta_{\vec{J}}=
\bigcup\{
\vec{\delta}_{\vec\al}:
\vec\al\in\vec{J}\,\}$.
Thus, $p_{\vec{J}}$ is a member of $\bP$.
Since
for each $\vec\al\in\vec{J}$,
$\ran(p_{\vec{J}}\re \vec{\delta}_{\vec\al})=\ran(p_{\vec\al})$,
it follows that
$p_{\vec{J}}\le p_{\vec\al}$ for each $\vec\al\in\vec{J}$.
\end{proof}

Now
we build an  sRc  tree $S\in [D,T]$ so that  the coloring $h$ will be  monochromatic on $r_{k+1}[B,S]^*$.
Recall that
$D=r_n(T)$.
Let $M=\{ m_j:j<\om\}$ be the strictly increasing enumeration of those integers $m> n$
such that
for each $F\in r_{m}[D,T]$, the coding node  in $\max(F)$  extends $s_d$.
The integers in $M$  represent the stages at which we  will use the forcing to find the next level  of $S$ so that the members of $r_{k+1}[B,S]^*$ will have the same $h$-color.

For each $i\le d$,
extend the node $s_i\in B$ to the node $t^*_i$.
Then extend each node
 in $\max(D^+)\setminus B$  to a node in $T\re l^*$.
 If one wishes to be concrete, take the leftmost extensions in $T$;   how the nodes in
$\max(D^+)\setminus B$ are extended makes no difference to the conclusion of the theorem.
 Set
\begin{equation}
U^*=\{t^*_i:i\le d\}\cup\{u^*:u\in D^+\setminus B\}.
\end{equation}
$U^*$
end-extends   $\max(D^+)$.
If $m_0=n+1$,
then $D\cup U^*$ is a member of $r_{m_0}[D,T]$.
In
 this case, let $U_{m_0}=D\cup U^*$, and
 let $U_{m_1-1}$ be any member of $r_{m_1-1}[U_{m_0},T]$.
 Notice that $U^*$ is  the only member of $\Ext_{U_{m_1-1}}(B)$, and it has $h'$-color $\varepsilon^*$.

Otherwise,
 $m_0>n+1$.
 In this case,
 take some  $U_{m_0-1}\in r_{m_0-1}[D,T]$ such that  $\max(U_{m+1})$ end-extends $U^*$,
 and
 notice that   $\Ext_{U_{m_0-1}}(B)$ is empty.
Now assume that  $j<\om$ and
 we have constructed $U_{m_j-1}\in r_{m_j-1}[D,T]$
  so that every member of $\Ext_{U_{m_j-1}}(B)$
 has $h'$-color $\varepsilon^*$.
Fix some  $V\in r_{m_j}[U_{m_j -1} ,T]$ and let $Z=\max(V)$.
We will extend the nodes in $Z$  to construct
$U_{m_j}\in r_{m_j}[U_{m_j-1},T]$ which is homogeneous for $h'$ in value $\varepsilon^*$.
This is done  by constructing
 the condition $q$, below, and then extending it to some $r \le q$ which decides all members of $\Ext_T(B)$ coming from the nodes in $\ran(r)$ have $h'$-color $\varepsilon^*$.

Let $q(d)$ denote  the coding  node in $Z$ and let $l_q=|q(d)|$.
For each $i<d$,
let  $Z_i$ denote the set of nodes in $Z\cap T_i$; this set  has  $2^{m_j-1}$ many nodes.
For each $i<d$,
take  a set $J_i\sse K_i$ of cardinality $2^{m_j-1}$
and label the members of $Z_i$ as
$\{z_{\al}:\al\in J_i\}$.
Let $\vec{J}$ denote $\prod_{i<d}J_i$.
By   Lemma \ref{lem.compat},
the set $\{p_{\vec\al}:\vec\al\in\vec{J}\}$ is compatible, as evidenced by the fact that
$p_{\vec{J}}:=\bigcup\{p_{\vec\al}:\vec\al\in\vec{J}\}$ is a condition in $\bP$.

Let
 $\vec{\delta}_q=\bigcup\{\vec{\delta}_{\vec\al}:\vec\al\in \vec{J}\}$.
For $i<d$ and $\al\in J_i$,
define $q(i,\al)=z_{\al}$.
It follows  that for each
$\vec\al\in \vec{J}$ and $i<d$,
\begin{equation}
q(i,\al_i)\contains t^*_i=p_{\vec\al}(i,\al_i)=p_{\vec{J}}(i,\al_i),
\end{equation}
and
\begin{equation}
q(d)\contains t^*_d=p_{\vec\al}(d)=p_{\vec{J}}(d).
\end{equation}
For   $i<d$ and $\delta\in\vec{\delta}_q\setminus
J_i$,
let $q(i,\delta)$ be the leftmost extension
 of $p_{\vec{J}}(i,\delta)$ in $T$ of length $l_q$.
Define
\begin{equation}
q=\{q(d)\}\cup \{\lgl (i,\delta),q(i,\delta)\rgl: i<d,\  \delta\in \vec{\delta}_q\}.
\end{equation}
This $q$ is a condition in $\bP$, and $q\le p_{\vec{J}}$.

To construct $U_{m_j}$,
take an $r\le q$ in  $\bP$ which  decides some $l_j$ in $\dot{L}_d$ for which   $h'(\dot{b}_{\vec\al}\re l_j)=\varepsilon^*$, for all $\vec\al\in\vec{J}$.
This is possible since for all $\vec\al\in\vec{J}$,
$p_{\vec\al}$ forces $h'(\dot{b}_{\vec\al}\re l)=\varepsilon^*$ for $\dot{\mathcal{U}}$ many $l\in \dot{L}_d$.
By the same argument as in creating the conditions $p_{\vec\al}$ to satisfy (3),
 we may assume that
 the nodes in the image of $r$ have length  $l_j$.
Since
$r$ forces $\dot{b}_{\vec{\al}}\re l_j=X(r,\vec\al)$
for each $\vec\al\in \vec{J}$,
and since the coloring $h'$ is defined in the ground model,
it follows that
$h'(X(r,\vec\al))=\varepsilon^*$ for each $\vec\al\in \vec{J}$.
Let $Y$ be the level set
consisting of the nodes $\{r(d)\}\cup \{r(i,\al):i<d,\ \al\in J_i\}$ along with
 a unique  node  $y_z$ in $T\re l_j$ extending  $z$, for each
 $z\in
Z\setminus (\{r(d)\}\cup \{r(i,\al):i<d,\ \al\in J_i\})$.
Then $Y$ end-extends $Z$.
Letting $U_{m_j}=U_{m_j-1}\cup Y$,
we see that $U_{m_j}$ is a member of $r_{m_j}[U_{m_j-1},T]$ such that
$h'$ has value $\varepsilon^*$ on $\Ext_{U_{m_j}}(B)$.
Let $U_{m_{j+1}-1}$ be any member of
$r_{m_{j+1}-1}[U_{m_j},T]$.
This completes the inductive construction.

Let $S=\bigcup_{j<\om}U_{m_j}$.
Then $S$ is a member of $[D,T]$ and
 for each $X\in\Ext_{S}(B)$,  $h'(X)=\varepsilon^*$.
Thus, $S$ satisfies the theorem.
\end{proof}

Corollary \ref{cor.A.4} follows immediately from  Theorem \ref{thm.matrixHL};
Case (a) handles  $k\ge 1$ and Case (b) handles $k=0$.
For the reader  interested in topological Ramsey spaces, we point out that this corollary states that
Axiom \bf A.4 \rm of Todorcevic's Axioms in Chapter 5, Section 1 of \cite{TodorcevicBK10}
holds for the space $\mathcal{T}$ of strong Rado coding trees.

\begin{cor}\label{cor.A.4}
Let $k<\om$, $A\in\mathcal{AT}_k$,  and  $T\in\mathcal{T}$ be given with $A\sse T$, and let $n=\depth_T(A)$.
For any subset $\mathcal{O}\sse r_{k+1}[A,T]$,
there is an $S\in [r_n(T),T]$ such that either
$\mathcal{O}\sse r_{k+1}[A,S]$ or else
$\mathcal{O}\cap
 r_{k+1}[A,S]=\emptyset$.
\end{cor}


\section{Borel sets of  strong  Rado coding  trees are completely Ramsey}\label{sec.MainThm}

In this section we prove Theorem \ref{thmmain}:
Given a Rado graph $\bR=(\om,E)$,
Borel subsets of the space $\mathcal{T}_{\bR}$ of strong Rado coding trees are   completely  Ramsey (see Definition \ref{defn.CR}).
The proof entails showing  that the collection of completely Ramsey subsets of $\mathcal{T}_{\bR}$
is a $\sigma$-algebra containing
 all open sets.
Showing that open sets are completely Ramsey   is accomplished by induction on the rank of Nash-Williams families, collections of finite sets which  determine basic open sets.
However, since the space $\mathcal{T}_{\bR}$  does not possess the  same  amalgamation property that the Baire space has, we will need to do induction on open sets extending members of $\widehat{\mathcal{AT}}$.
This is the reason the  broader statement  in Theorem \ref{thm.matrixHL} was proved,
rather than simply Corollary \ref{cor.A.4}.
We will apply Theorem \ref{thmmain} in the next section to prove the Main Theorem.

As usual,
fix a Rado graph  $\mathbb{R}=(\om,E)$, and let  $\bS$ denote $\bS_{\bR}$ and $\mathcal{T}$ denote $\mathcal{T}_{\bR}$.
Recall  the definitions of
$\widehat{\mathcal{AT}}$ and
$\widehat{\mathcal{AT}}(T)$,
 in
equations  (\ref{eq.AThat}) and   (\ref{eq.AThatT}), respectively, as well as Definition \ref{defnhatAT}.
Given $B\in\widehat{\mathcal{AT}}$ and  $T\in\mathcal{T}$
an important property of the set $[B,T]^*$ is that
it is open in the Ellentuck topology on $\mathcal{T}$:
If $B$ is in $\mathcal{AT}_k$ for some $k$,
then  the set  $[B,T]^*$ is the union of
$[B,T]$ along with
 all $[C,T]$, where
$C \in\mathcal{AT}_k$ and $\max(C)$ end-extends $\max(B)$.
If $B$ is in $\widehat{\mathcal{AT}}$ but not in $\mathcal{AT}$,
then
 letting  $k$ be the least integer for which there is some $C\in\mathcal{AT}_k$  with $\max(C)\sqsupset \max(B)$,
 we see that
$[B,T]^*$ equals the union of  all $[C,T]$, where $C\in\mathcal{AT}_{k}$ and $C\sqsupseteq B$.
For the same reasons, the set $[B, \bS]^*$ is open in the  metric topology on $\mathcal{T}$.

\begin{defn}\label{defn.NWfamily}
A subset $\mathcal{F}\sse\mathcal{AT}$  is said to  have the {\em Nash-Williams property} if  for any two distinct members $F,G\in\mathcal{F}$,
neither is an initial segment of the other.
\end{defn}

A Nash-Williams family  $\mathcal{F}$ determines the  metrically open  set
\begin{equation}
\mathcal{O}_{\mathcal{F}}
=
\bigcup_{F\in\mathcal{F}}[F,\bS].
\end{equation}
Conversely,  to each open  set  $\mathcal{O}\sse\mathcal{T}$ in the metric topology there corresponds a Nash-Williams family
$\mathcal{F}(\mathcal{O})$ by defining
$F\in\mathcal{AT}$ to be a member of
$\mathcal{F}(\mathcal{O})$
if and only if
$[F,\bS]\sse \mathcal{O}$ and if $G$ is any proper initial segment of $F$, then $[G,\bS]\not\sse\mathcal{O}$.


Given $\mathcal{F}\sse \mathcal{AT}$ and
$B\in\widehat{\mathcal{AT}}$, define
\begin{equation}
\mathcal{F}_B=\{F\in\mathcal{F}:  \exists k\,(\max(r_k(F))\sqsupseteq \max(B))\}.
\end{equation}
In particular, if $\mathcal{F}\sse r[B,\bS]^*$, then $\mathcal{F}_B=\mathcal{F}$.
If  $\mathcal{F}$ is  a Nash-Williams family, then  $B\in\mathcal{F}$ if and only if $\mathcal{F}_B=\{B\}$.
Given $T\in\mathcal{T}$, let
\begin{equation}
\mathcal{F}|T=\{F\in\mathcal{F}:F\in\mathcal{AT}(T)\}.
\end{equation}
With this notation, notice that $\mathcal{F}_B|T=\mathcal{F}\cap r[B,T]^*$, for any $B\in\widehat{\mathcal{AT}}$.

For  $F\in\mathcal{AT}$,
recall from Subsection  \ref{subsec.tRs}
that
$|F|$ denotes the $k$ for which $F\in\mathcal{AT}_k$.
Given a set $\mathcal{F}\sse \mathcal{AT}$,
let
\begin{equation}
\tilde{\mathcal{F}}=\{r_k(F):F\in\mathcal{F}\mathrm{\ and\ }k\le |F|\},
\end{equation}
 and note that  $\tilde{\mathcal{F}}\sse\mathcal{AT}$.
If $\mathcal{F}$ is a Nash-Williams family, then  $\mathcal{F}$ consists of the $\sqsubseteq$-maximal members of $\tilde{\mathcal{F}}$.

\begin{defn}\label{defn.front}
Suppose $T\in\mathcal{T}$ and $B\in\widehat{\mathcal{AT}}(T)$.
We say that a family $\mathcal{F}\sse r[B,T]^*$ is a {\em front on $[B,T]^*$} if
$\mathcal{F}$ is a Nash-Williams family and
for each $S\in [B,T]^*$,
there is some $C\in\mathcal{F}$ such that $C\sqsubset S$.
\end{defn}

Notice that a front $\mathcal{F}$ on $[B,T]^*$ determines a collection of disjoint (Ellentuck) basic open sets $[C,T]$, $C\in\mathcal{F}$, whose union is exactly $[B,T]^*$.

\begin{assumption}\label{assumption.important}
Let $T\in\mathcal{T}$ be  fixed, and let $D=r_d(T)$ for some $d<\om$.
Given $A\in \widehat{\mathcal{AT}}(T)$ with $\max(A)\sse\max(D)$.
Recall that
 $A^+$ denotes the union of $A$  with the set of   immediate extensions in    $\widehat{T}$ of the members of $\max(A)$.
Let $B$ be a member of $\widehat{\mathcal{AT}}$ such that
 $\max(A)\sqsubset \max(B)\sse\max(A^+)$.
We consider two  cases  for triples $(A,B,k)$:

\begin{enumerate}
\item[]
\begin{enumerate}
\item[\bf{Case (a).}]
$k\ge 1$,
$A\in \mathcal{AT}_k(T)$,
and
 $B=A^+$.
\end{enumerate}
\end{enumerate}

\begin{enumerate}
\item[]
\begin{enumerate}
\item[\bf{Case (b).}]
$A$ has at least one node, and
each member of $\max(A)$ has exactly one extension in $B$.
Let $k$ be the integer satisfying $2^k=\mathrm{card}(\max(A))$.
\end{enumerate}
\end{enumerate}
\end{assumption}

\begin{thm}\label{thm.GalvinNW}
Given $T\in\mathcal{D}$, $(A,B,k)$, $d=\depth_T(A)$, and $D=r_d(T)$ as in  Assumption \ref{assumption.important},
let
$\mathcal{F}\sse r[B,T]^*$ be  a Nash-Williams family.
Then   there is an $S\in [D,T]$ such that
either $\mathcal{F}|S$ is a front on $[B,S]^*$
or else
 $\mathcal{F}|S=\emptyset$.
\end{thm}

\begin{proof}
Recall that  the notation $S\le T$ means that  $S\in [\emptyset, T]$.
We say that 
$S\le T$ {\em accepts} $C\in r[B,S]^*$ if  $\mathcal{F}_C| S$ is a front on $[C,S]$.
We say that 
$S$ {\em widely-rejects} (w-rejects) $a$ if either
\begin{enumerate}
\item[(a)]
$C\not\in r[B, S]^*$; or
\item[(b)]
$C\in  r[B, S]^*$ and $\forall P\in [\depth_S(C),S]\ \exists Q\in[C,P]\ \forall n(r_n(Q)\not\in\mathcal{F})$.
\end{enumerate}
We say that $S$ {\em decides} $C$ if either $S$ accepts $C$ or else $S$ w-rejects $C$.
For $n\in\om$,  let  $[n,S]$ denote $[r_n(S),S]$.

\begin{fact}\label{fact.1}
If $S$ accepts $C$, then so does each $P\le S$ with $C\in \mathcal{AD}(P)$.
If $S$ w-rejects $C$, then 
either $C\not\in r[B, S]^*$ and  every $P
\le S$ also w-rejects  $C$, or else $C\in r[B,S]^*$ and every  $P\in [\depth_S(C),S]$  w-rejects $C$.
\end{fact}

\begin{proof}
Suppose $S$ accepts $C$ and $P\le S$ with $C\in\mathcal{AD}(P)$.
Since $\mathcal{F}_C|S$ is a front on $[C,S]$, it follows that $\mathcal{F}_C|P$ is a front on $[C,P]$.
Hence $P$ accepts $C$.

Suppose $S$ w-rejects $C$.  If $C\not\in
r[B,S]^*$,
 then
also for each $P\le S$, 
 $C\not\in r[B,S]^*$  and 
hence
 $P$ w-rejects $C$.
Otherwise,  $C\in   r[B,S]^* $.
Let $n=\depth_S(C)$ and suppose
$P\in [n,S]$. 
Since  $S$ w-rejects $C$,   for each $Q\in [n,S]$   there is an $R\in [n,Q]$ such that for all $m$, $r_m(R)\not\in \mathcal{F}$.
Note that  $P\in [n,S]$ implies $[n,P]\sse[n,S]$; 
so 
 for each $Q\in [n,P]$ there is an $R\in [a,Q]$ such that  for all $m$, $r_m(Q)\not\in \mathcal{F}$. Therefore, $P$ w-rejects $C$.
\end{proof}

\begin{lem}\label{lem.A}
Given $C\in r[B,S]^*$  and $n=\depth_S(C)$,
either $\exists P\in [n,S]$ which  w-rejects $C$, or else $\forall P\in[n,S]\ \exists Q\in[n,P]$ which accepts $C$.
\end{lem}

\begin{proof}
Suppose there is no $P\in[n,S]$ which w-rejects $C$.
Then $\forall P\in[n,S]$, 
\begin{equation}
\exists Q\in [n,P]\ \forall X\in [C,Q]\ \exists m(r_m(X)\in \mathcal{F}).
\end{equation}
Thus,  for all $P\in [n,S]$ there is a $Q\in [n,P]$ such that $\mathcal{F}_C|Q$ is a front on $[C, Q]$;
that is, $Q$  accepts $C$. 
\end{proof}

\begin{fact}\label{fact.2}
\begin{enumerate}
\item[(a)]
For each $C\in r[B,T]^*$, there is an $S\in[\depth_T(C),T]$ which decides $C$.
\item[(b)]
 If $C\in r[B,T]^*$, then $S\in [B,T]^*$ with $C\in\mathcal{AD}(S)$ accepts $C$ if and only if $S$ accepts each $F\in r_{|C|+1}[C,S]$.
\end{enumerate}
\end{fact}

\begin{proof}
For (a),  let  $n=\depth_S(C)$.
By  Lemma \ref{lem.A}, either there is an $S\in [n,T]$ which w-rejects $C$,
or else there is an $S\in [n,T]$ which accepts $C$.

For (b), given the hypotheses,
$S$ accepts $C$ iff $\mathcal{F}_C|S$ is a front on $[C,S]$ iff for each $F\in r_{|C|+1}[C,S]$, 
$\mathcal{F}_C|S$ is a front on $[F,S]$
iff $S$ accepts each  $F\in r_{|C|+1}[C,S]$.
\end{proof}

Recall that $B\in\widehat{\mathcal{AD}}$, but is  not necessarily a member of $\mathcal{AD}$.
We shall say 
 that $S$ {\em accepts} $B$ if $S$ accepts
$F$ for all $F\in r_{k+1}[B,S]^*$.

\begin{fact}\label{fact.claim.1}
If $S\in[B,T]^*$  accepts 
$B$, then $\mathcal{F}_B|S$ is a front on $[B,S]^*$.
\end{fact}

\begin{proof}
For  each 
 $C\in r_{k+1}[B,S]^*$,   $S$ accepts $C$ implies that 
$\mathcal{F}_C|S$ is a front on $[C,S]$.
Since  $[B,S]^*=\bigcup\{[C,S]:C\in  r_{k+1}[B,S]^*\}$, it follows 
 that   $\mathcal{F}|S=\bigcup\{\mathcal{F}_C|S:C\in r_{k+1}[B,S]^*\}$, which is a front on $[B,S]^*$.
\end{proof}

\begin{lem}\label{lem.decides}
There is an $S\in[d,T]$ which 
decides each $C$ in $r[B,S]^*$.
\end{lem}

\begin{proof}
By finitely many applications of 
Fact \ref{fact.2}, 
we obtain a $T_1\in [d+1,T]$ such that $T_1$ decides each 
$C\in r[B,T]^*$  with 
 $C\sse r_{d+1}(T)$.
Given $T_i$,
by finitely many applications of 
Fact \ref{fact.2}, 
we obtain a $T_{i+1}\in [d+i+1,T_i]$ such that $T_{i+1}$ decides each 
$C\in r[B,T_i]^*$  with 
 $C \sse  r_{d+i+1}(T_i)$.
Let $S=\bigcup_{i=1}^{\infty} r_{d+i}(T_i)$, which is the same as  $\bigcup_{i=1}^{\infty} r_{d+i+1}(T_i)$.
Then $S\in [d,T]$
(in fact, $S\in[d+1,T]$)
 and 
for  $C\in r[B,S]^*$,
$T_i$ decides $C$, where $i$ is the least index satisfying
 $C  \sse  r_{d+i}(T_i)$.
Since $S\in [d+i,T_i]$, it follows that $S$ decides $C$ in the same way that $T_i$ does.
Thus, $S$ decides every  $C\in r[B,S]^*$.
\end{proof}

Now we finish the proof of the theorem. 
Take $S$ as in Lemma \ref{lem.decides} and define a coloring
 $f:r[B,S]^*\ra 2$ by 
$f(C)=0$ if $S$ accepts $C$ and $f(C)=1$ if $S$ w-rejects $C$.
By the Extended Pigeonhole Principle, Theorem \ref{thm.matrixHL}, there is a $P\in [d,S]$ for which $f$ is monochromatic on $r_{k+1}[B,P]^*$.
Now if $f$ has color $0$ on this set, then
$P$ accepts $B$ and 
 by Fact \ref{fact.claim.1},  $\mathcal{F}|P$ is a front on $[B,P]^*$.

Otherwise, $f$ has color $1$ on 
$r_{k+1}[B,P]^*$  so  $P$ w-rejects each member of $r_{k+1}[B,P]^*$. 
Let $P_0=P$.
Apply Theorem \ref{thm.matrixHL} finitely many (possibly $0$) times, to obtain some $P_{1}\in[d+1,P_0]$ such that 
for each $C \in r[B,P_{1}]^*$ with $C
\sse  r_{d+1}(P_{0})$,
all members of 
$r_{|C|+1}[C,P_{1}]^*$ have the same $f$-color.
Since such a $C$ is necessarily in $r_{k+1}[B,P_0]^*$ and  $P_0$ w-rejects $C$, Fact \ref{fact.2} implies that this $f$-color must be $1$.

For 
 $i\ge 1$, 
 we have the following  the induction hypothesis:
 $P_i\in[d+i,P_{i-1}]$  and 
for 
each   $C\in r[B,P_{i-1}]^*$ with $C\sse  r_{d+i}(P_{i-1})$,
$P_i$ w-rejects all members of $r_{|C|+1}[C,P_i]$.
Apply Theorem \ref{thm.matrixHL}
finitely many times to obtain a $P_{i+1}\in [d+i+1,P_i]$ such that $f$ is monochromatic on $r_{|C|+1}[C,P_{i+1}]$ for each $C\in r[B,P_i]^*$ with $C\sse  r_{d+i+1}(P_i)$.
Fix a $C\in r[B,P_i]^*$ with $C\sse r_{d+i+1}(P_i)$.
If  $|C|=k+1$ then $P_{i+1}$ w-rejects
 $C$, since $C\in r_{k+1}[B,P]^*$ and $P_{i+1}\in[B,P]^*$.
Suppose now that $|C|>k+1$.
By   the induction hypothesis, $P_i$ w-rejects $C$ since $C\in r_{|F|+1}[F,P_i]$, where 
$F=r_{|C|-1}(C)\sse r_{d+i}(P_{i-1})$.
Now if the  $f$-color on $r_{|C|+1}[C, P_{i+1}]$ is  $0$, then $P_{i+1}$ accepts $C$ by Fact \ref{fact.2}, a contradiction. 
Hence, $f$ has color $1$ on $r_{|C|+1}[C, P_{i+1}]$; 
in particular, $P_{i+1}$ w-rejects each member of $r_{|C|+1}[C, P_{i+1}]$.

Let $Q=\bigcup_{i=1}^{\infty} r_{d+i}(P_i)$.
Then $Q$ w-rejects each member of $r[B,Q]^*$.
By definition of w-rejects,
for each  $C\in  r[B,Q]^*$,
\begin{equation}\label{eq.dagger}
\forall R\in[\depth_Q(C),Q]\ 
\exists X\in[C,R]\ \forall n(r_n(X)\not\in\mathcal{F})
\end{equation}
Suppose toward a contradiction that there is an $C\in\mathcal{F}|Q$.
Then for all $X\in [C,Q]$, $r_{|C|}(X)=C\in\mathcal{F}$.
So  $Q\in [\depth_{Q}(C),Q]$  and   for all $X\in [C,Q]$, $\exists n(r_n(X)\in\mathcal{F})$.
But this contradicts  (\ref{eq.dagger}).
Thus $\mathcal{F}|Q$ must be empty.
\end{proof}


\begin{defn}\label{defn.CR}
Let  $\mathcal{X}$ be a subset of $\mathcal{T}$.
We say that $\mathcal{X}$ is {\em Ramsey}
if for each $T\in\mathcal{T}$ there is  an $S\le T$ such that either $\mathcal{X}\sse [\emptyset,S]$ or else $\mathcal{X}\cap  [\emptyset,S]=\emptyset$.
 $\mathcal{X}$ is said to be  {\em completely Ramsey (CR)}  if for each $C\in\mathcal{AT}$ and each $T\in \mathcal{T}$,
there is an $S\in [C,T]$ such that either $[C,S]\sse \mathcal{X}$ or else $[C,S]\cap\mathcal{X}=\emptyset$.
For this article, we introduce  additional terminology:
  $\mathcal{X}$  is {\em CR$^*$}
if for each quadruple $T,A,B,D$ as in Assumption \ref{assumption.important},
there is an $S\in [D,T]$ such that either $[B,S]^*\sse \mathcal{X}$ or else $[B,S]^*\cap\mathcal{X}=\emptyset$.
\end{defn}



\begin{rem}
Theorem \ref{thm.GalvinNW} 
 shows that metrically open sets are completely Ramsey.
Importantly, it proves  the stronger statement that metrically  open sets are
CR$^*$.
This
 stronger statement will be used to get around the lack of amalgamation  (Todorcevic's Axiom
 \bf A.3\rm(b))  for $(\mathcal{T},\le,r)$, to prove that Borel sets are completely Ramsey, and in fact, even CR$^*$.
\end{rem}

\begin{lem}\label{lem.complements}
Complements of CR$^*$ sets are CR$^*$.
\end{lem}

\begin{proof}
Suppose $\mathcal{X}\sse \mathcal{T}$ is CR$^*$.
Given $T,B,D$ as in Assumption \ref{assumption.important},
 by definition of CR$^*$, there is an $S\in[D,T]$ such that either $[B,S]^*\sse\mathcal{X}$ or else
 $[B,S]^*\cap\mathcal{X}=\emptyset$.
Letting $\mathcal{Y}=\mathcal{T}\setminus \mathcal{T}$,
the complement of $\mathcal{X}$,
we see that either
$[B,S]^*\cap\mathcal{Y}=\emptyset$ or else $[B,S]^*\sse\mathcal{Y}$.
\end{proof}

In the rest of this section,
given $T\in\mathcal{T}$,  assume that
 $[\emptyset,T]$ inherits the
 subspace topology from  $\mathcal{T}$ with the metric topology.
 Thus, the basic metrically open sets of $[\emptyset,T]$ are of the form
 $[C,T]$, where $C\in \mathcal{T}(T)$.
The next two lemmas set up for Lemma \ref{lem.ctblU}
that countable unions of CR$^*$ sets are CR$^*$.

\begin{lem}\label{lem.GP8}
Suppose $\mathcal{X}\sse\mathcal{T}$ is CR$^*$.
Then for each $T\in\mathcal{T}$ and each $D\in\mathcal{AT}(T)$,
there is an $S\in [D,T]$ such that $\mathcal{X}\cap [\emptyset,S]$ is metrically  open   in $[\emptyset,S]$.
\end{lem}

\begin{proof}
Fix $T\in\mathcal{T}$ and  $D\in\mathcal{AT}(T)$.
Let $\lgl (A_j,B_j):j<\tilde{j}\rgl$ be an enumeration of the pairs $(A,B)$ satisfying Assumption
\ref{assumption.important} for   $T$ and $D$.
Notice that  $\bigcup_{j<\tilde{j}}[B_j,T]^*=[\emptyset,T]$.
Let $T_{-1}=T$.
For $j<\tilde{j}$, given $T_{j-1}$,
by the definition of CR$^*$  we may take  some $T_j\in [D,T_{j-1}]$ such that either $[B_j,T_j]^*\sse\mathcal{X}$ or else $\mathcal{X}\cap [B_j,T_j]^*=\emptyset$.

Let $S=T_{\tilde{j}-1}$.
Then $S\in [D,T]$, and
for each $j<\tilde{j}$,
$[B_j,S]^*\sse [B_j,T_j]^*$.
Notice that
$[\emptyset,S]=\bigcup_{j<\tilde{j}}[B_j,S]^*$.
  (If $B_j\not\sse S$, then $[B_j,S]^*=\emptyset$.)
Hence,
\begin{equation}
\mathcal{X}\cap [\emptyset,S]=\bigcup_{j<\tilde{j}} (\mathcal{X}\cap [B_j,S]^*).
\end{equation}
For $j<\tilde{j}$,
if $ [B_j,T_j]^*\sse \mathcal{X}$
then
 $\mathcal{X}\cap [B_j,S]^*= [B_j,S]^*$;
and if $\mathcal{X}\cap  [B_j,T_j]^*=\emptyset$
 then
  $\mathcal{X}\cap [B_j,S]^*=\emptyset$.
Thus,
\begin{equation}
\mathcal{X}\cap[\emptyset,S]=\bigcup_{j\in J}[B_j,S]^*,
\end{equation}
where $J=\{j< \tilde{j}: [B_j,T_j]^*\sse \mathcal{X}\}$.
As each $[B_j,S]^*$ is metrically open in the subspace $[\emptyset,S]$,
$\mathcal{X}\cap[\emptyset,S]$  is  also metrically open in the subspace $[0,S]$.
\end{proof}

\begin{lem}\label{lem.GP9}
Suppose $\mathcal{X}_n$, $n<\om$, are CR$^*$ sets.
Then for each $T\in\mathcal{T}$ and each $D\in\mathcal{AT}(T)$,
there is an $S\in [D,T]$ such that for each $n<\om$, $\mathcal{X}_n\cap[\emptyset,S]$ is metrically open in
$[\emptyset,S]$.
\end{lem}

\begin{proof}
Suppose $\mathcal{X}_n$, $n<\om$,  are  CR$^*$ sets.
Since  $\mathcal{X}_0$ is CR$^*$,  Lemma \ref{lem.GP8} implies that
there is an $S_0\in [r_d(T),T]$ such that $\mathcal{X}_0\cap [\emptyset,S_0]$ is metrically open in the subspace topology on
$[\emptyset,S_0]$.
Let $\mathcal{O}_0\sse\mathcal{T}$ be a metrically open set satisfying
$\mathcal{X}_0\cap [\emptyset,S_0]=\mathcal{O}_0\cap[\emptyset,S_0]$.
In general, given $i<\om$ and
$S_i\in [r_{d+i}(T),T]$,
by  Lemma \ref{lem.GP8} there is some $S_{i+1}\in
[r_{d+i+1}(S_i),S_i]$
and some
 metrically open
 $\mathcal{O}_i\sse\mathcal{T}$
such that
$\mathcal{X}_{i}\cap [\emptyset,S_{i}]=\mathcal{O}_i\cap[\emptyset,S_i]$.
Let $S=\bigcup_{i<\om}r_{d+i}(S_i)$.
Then $S$ is a member of $[r_d(T),T]$.
Letting $S_{-1}=T$,
note that
 $S\in [r_{d+i}(S_i),S_{i-1}]$ for each $i<\om$.
It follows that
$\mathcal{X}_{i}\cap [\emptyset,S]=\mathcal{O}_i\cap[\emptyset,S]$; hence $\mathcal{X}_{i}\cap [\emptyset,S]$ is metrically open in $[\emptyset,S]$.
\end{proof}

\begin{lem}\label{lem.ctblU}
Countable unions of CR$^*$ sets are CR$^*$.
\end{lem}

\begin{proof}
Suppose $\mathcal{X}_n$, $n<\om$,  are  CR$^*$ subsets of $\mathcal{T}$, and
let $\mathcal{X}=\bigcup_{n<\om}\mathcal{X}_n$.
Let  $T,B,D,d,k$ be  as in Assumption \ref{assumption.important}.
We claim that  there is some $U\in [D,T]^*$ such that $[B,U]^*\sse\mathcal{X}$ or $[B,U]^*\cap\mathcal{X}=\emptyset$.

By Lemma \ref{lem.GP9}, there is
an $S\in [D,T]$ such that for each $n<\om$,
$\mathcal{X}_n\cap[\emptyset,S]$ is metrically open in   $[\emptyset,S]$.
Thus, $\mathcal{X}\cap[\emptyset,S]$ is metrically  open in $[\emptyset,S]$,
 so $\mathcal{X}\cap[\emptyset,S]=\mathcal{O}\cap [\emptyset,S]$ for some metrically open set $\mathcal{O}\sse \mathcal{T}$.
Relativizing to $[\emptyset,S]$,
Theorem  \ref{thm.GalvinNW}  implies that  $\mathcal{O}$ is
 CR$^*$ (in $[\emptyset,S]$).
Since $r_d(S)=D$,
by definition of CR$^*$  there is some $U\in [D,S]$ such that
either $[B,U]^*\sse \mathcal{O}$ or else
$[B,U]^*\cap\mathcal{O}=\emptyset$.
Therefore,
either
\begin{equation}
[B,U]^*  =[B,U]^*\cap[\emptyset,S]\sse \mathcal{O}\cap[\emptyset,S]=\mathcal{X}\cap[\emptyset,S],
\end{equation}
or else
\begin{align}
[B,U]^*\cap\mathcal{X}&=
[B,U]^*\cap[\emptyset,S]\cap\mathcal{X}\cr
&\sse [B,U]^*\cap[\emptyset,S]\cap\mathcal{O}\cr
&\sse [B,U]^*\cap\mathcal{O}
=\emptyset.
\end{align}
Thus, $\mathcal{X}$ is CR$^*$.
\end{proof}

\begin{thm}\label{thm.best}
The collection of CR$^*$ subsets of $\mathcal{T}$ contains all Borel subsets of $\mathcal{T}$.
In particular,
Borel subsets of the space $\mathcal{T}$ of strong Rado coding trees are completely Ramsey.
\end{thm}

\begin{proof}
This follows from Theorem \ref{thm.GalvinNW} and  Lemmas  \ref{lem.complements} and
\ref{lem.ctblU}.
\end{proof}

Theorem  \ref{thmmain}  is a special case of  the second half of Theorem \ref{thm.best}.

\begin{rem}\label{rem.Q}
We end this section with a remark about the infinite
 dimensional Ramsey theory  of  the rationals.
Extend now the lexicographic order on $\bS$ as follows:
For $s,t\in \bS$, define
$s<_{\mathrm{lex}} t$ exactly when
one of the following hold:
(a) $s$ and $t$ are incomparable  and
$s(|s\wedge t|)<t(|s\wedge t|)$;
 (b) $s\subset t$ and $t(|s|)=1$; or
 (c) $t\subset s$ and $s(|t|)=0$.
This defines a total order on $\bS$ so that $(\bS,<_{\mathrm{lex}})$ is order isomorphic to $(\mathbb{Q},<)$.
Using this isomorphism, say $\varphi:(\bS,<_{\mathrm{lex}})\ra (\mathbb{Q},<)$,
 Milliken's Theorem  \ref{thm.M} provides infinite dimensional Ramsey theorem  for the space of all subsets of $\mathbb{Q}$ which are $\varphi$-images of strong subtrees of $\bS$.
 This handles only one strong similarity type of subcopies of $\mathbb{Q}$, though.

It is useful to note that since the  coding nodes in the trees in $\mathcal{T}$ are dense,
they also represent  the rationals, using the  extended lexicographic order $<_{\mathrm{lex}}$  just defined.
Thus,  Theorem \ref{thm.best} provides an infinite dimensional Ramsey theorem for the rationals.
Similarly to the constraints provided by the work in \cite{Laflamme/Sauer/Vuksanovic06} for the Rado graph,
the work of Devlin \cite{DevlinThesis} on the big Ramsey degrees of the rationals provides constraints, again in terms of strong similarity types, for the infinite dimensional Ramsey theory of the rationals.
\end{rem}


\section{The Main Theorem}\label{sec.MT}

We now prove the Main Theorem.
Recall the homeomorphism  $\theta:\mathcal{R}(\bR)\ra\mathcal{T}_{\bR}$ defined at the end of Section
\ref{section.strongRadotrees}
by $\theta(\bR')=T_{\bR'}$,
for $\bR'\in\mathcal{R}(\bR)$.
Note that  given  a Borel subset
$\mathcal{X}\sse\mathcal{R}(\bR)$, the set
$\theta[\mathcal{X}]$ is a Borel subset of $\mathcal{T}_{\bR}$.
The subspace $\mathcal{R}(\bR)$ of the Baire space inherits the Ellentuck topology, refining the metric topology on $[\om]^{\om}$:
Given $\bR'\in\mathcal{R}(\bR)$ and $n\in \om$,
define $r_n(\bR')$ to be the subgraph of $\bR'$ induced on the first $n-1$ vertices of $\bR'$.
Let
\begin{equation}\label{eq.r_nRado}
\mathcal{AR}
=\{r_n(\bR'):\bR'\in\mathcal{R}(\bR)\mathrm{\ and\ } n<\om\}.
\end{equation}
For $\mathbb{F}\in \mathcal{AR}$ and $\bR'\in\mathcal{R}(\bR)$,
write $\mathbb{F} \sqsubset \bR'$ if and only if $\mathbb{F}=r_n(\bR')$ for some $n$.
Define
\begin{equation}
[\mathbb{F},\bR']=\{\bR''\in\mathcal{R}(\bR'): \mathbb{F}\sqsubset \bR''\}.
\end{equation}
We say that a set $\mathcal{X}\sse\mathcal{R}(\bR)$ is {\em completely Ramsey}
if for any $\mathbb{F}\in\mathcal{AR}$ and $\bR'\in\mathcal{R}(\bR)$,
there is some $\bR''\in [\mathbb{F},\bR']$ such that either
$[\mathbb{F},\bR'']\sse \mathcal{X}$
 or else
$[\mathbb{F},\bR'']\cap \mathcal{X}=\emptyset$.

\begin{mainthm}\label{thm.main}
Let $\mathbb{R}=(\om,E)$ be the Rado graph. Then every Borel subset  $\mathcal{X}\sse \mathcal{R}(\bR)$ is completely Ramsey.
In particular,
if $\mathcal{X}\sse\mathcal{R}(\mathbb{R})$ is Borel,
then for each $\mathbb{R}'\in\mathcal{R}(\mathbb{R})$,
there is a  Rado graph
 $\mathbb{R}''\in\mathcal{R}(\mathbb{R}')$  such that
$\mathcal{R}(\mathbb{R}'')$ is
either  contained in  $\mathcal{X}$,
or else is disjoint from $\mathcal{X}$.
\end{mainthm}

\begin{proof}
Let $\mathcal{X}$ be a Borel subset of
$\mathcal{R}(\bR)$, and suppose $\mathbb{F}\in\mathcal{AR}$ and $\bR'\in\mathcal{R}(\bR)$.
If $[\mathbb{F},\mathbb{R}']=\emptyset$ then
we are done,
so assume now that there is some $\bR''\in
[\mathbb{F},\mathbb{R}']$.
Then $\mathbb{F}$ is an initial segment of $\bR''$.
Let $n$ be the integer such that $\mathbb{F}=r_n(\bR'')$,
and let $A=r_n(T_{\bR''})$.
Since
$\mathbb{F}$ is an initial segment of $\bR''$, it follows that
 $[A,T_{\bR''}]$ is the $\theta$-image of
$[\mathbb{F},\mathbb{R}'']$.

Let $\mathcal{Y}=\theta[\mathcal{X}]$ and apply Theorem \ref{thm.best} to obtain an $S\in [A,T_{\bR''}]$ such that either $[A,S]\sse\mathcal{Y}$ or else
$[A,S]\cap\mathcal{Y}=\emptyset$.
Since  $\theta^{-1}$ is a homeomorphism from
$\mathcal{T}_{\bR}$ to $\mathcal{R}(\bR)$,
we have that
either $\theta^{-1}[A,S]\sse\mathcal{X}$ or else  $\theta^{-1}[A,S]\cap\mathcal{X}=\emptyset$.
Notice that
\begin{align}
\theta^{-1}[A,S] &=\{\theta^{-1}(S'):S'\in[ A,S]\}\cr
&=\{\mathbb{G}_{S'}:S'\in[ A,S]\}\cr
&=[\mathbb{G}_A,\mathbb{G}_S]\cr
&=[\mathbb{F}, \mathbb{G}_S].
\end{align}
Thus, we have found a $\mathbb{G}_S\in [\mathbb{F},\bR']$ such that
either
$[\mathbb{F}, \mathbb{G}_S]\sse\mathcal{X}$
or else
$[\mathbb{F}, \mathbb{G}_S]\cap\mathcal{X}=\emptyset$.
 The case where  $\mathbb{F}$ is the empty graph (no vertices) yields  the second half of the theorem.
\end{proof}


\section{Concluding remarks and further directions}\label{section.end}

In this paper, we have proved that Borel subsets of certain closed  spaces of Rado graphs are
completely
Ramsey.
This is a Rado graph analogue of the Galvin-Prikry Theorem \ref{thm.GP} for the Baire space.
As pointed out in Section \ref{section.strongtrees},
the work of Laflamme, Sauer, and Vuksanovic in \cite{Laflamme/Sauer/Vuksanovic06} necessitate that we restrict our spaces to collections of Rado graphs with the same strong similarity type.
We now point out three areas for improvement on the results of this paper.

Firstly, we would like to have an analogue of Ellentuck's Theorem; that is, we would like to further  obtain a result showing that all subsets with the property of Baire in the Ellentuck topology on $\mathcal{T}_{\bR}$ are completely Ramsey.
There is a serious breaking point when trying to adjust the methods of Ellentuck to the setting of strong Rado coding trees, which seems very closely tied with not  having an amalgamation property like the Baire space does,
this property  being made concrete in
the axiom \bf A.3\rm(b) of Todorcevic.
It does not seem possible to develop the `combinatorial forcing' method used by Nash-Williams, Galvin and Prikry, and finally Ellentuck without having this strong amalgamation property.
Thus, either new methods will be necessary, or it may be that the spaces $\mathcal{T}_{\bR}$ might have a subset with the property of Baire in the Ellentuck topology which is not Ramsey.
We leave this as an open problem.

\begin{question}
Is every subset of  $\mathcal{T}_{\bR}$ with the property of Baire with respect to   the Ellentuck topology Ramsey?
\end{question}

If the answer is no, then this would decide the following fundamental question in the positive.

\begin{question}
Is there a topological Ramsey space which does not satisfy Todorcevic's
 Axiom \bf A.3\rm(b)?
\end{question}

Secondly, we would like to extend  the Main Theorem to
apply  to any strong similarity type of a Rado graph.
The work in
\cite{Sauer06} and
\cite{Laflamme/Sauer/Vuksanovic06} on finding the big Ramsey degrees for the Rado graph
has the important feature that, at the end of the applications of Milliken's Theorem, they take a strongly diagonal antichain of nodes in $\bS$ which codes the Rado graph. It is within this antichain that they prove that the number of strong similarity types  representing a given finite graph is the upper bound for the big Ramsey degree (see \cite{Sauer06}) as well as the lower bound (see \cite{Laflamme/Sauer/Vuksanovic06}).
By virtue of how we defined $\mathbb{S}_{\bR}$, given a Rado graph $\bR=(\om,E)$,
one sees that the coding nodes in $\mathbb{S}_{\bR}$, and hence in any member of $\mathcal{T}_{\bR}$, are dense in the tree.
This aids in the proofs, especially with the forcing arguments.
However, we would prefer a theorem of the following sort:  Given a strongly diagonal antichain $A$ of nodes coding the Rado graph,
the space of all sub-antichains with the same strong similarity type form a space in which Borel sets are Ramsey.
This seems achievable and is the subject of ongoing work.

In tandem with this, we come  to the third area for improvement, which Todorcevic mentioned to the author at the 2019 Luminy Workshop in Set Theory:
namely, that the ``correct'' infinite dimensional Ramsey theorem should recover the big Ramsey degrees for the Rado graph.
The results in the present paper 
recover upper bounds for the big Ramsey degrees by
using initial segments of the  Rado coding trees as envelopes and applying Main Theorem.
However, this approach does  not recover the lower bounds proved 
by Laflamme, Sauer, and Vuksanovic in 
\cite{Laflamme/Sauer/Vuksanovic06}.
 So, this paper may be thought of as providing a reasonable  answer to  the stated question in \cite{Kechris/Pestov/Todorcevic05}, but 
 not recovering everything in 
 the  intended question.
 We hope that the work in this paper will pave the way.



\bibliographystyle{amsplain}
\bibliography{references}

\end{document}